\documentclass[smallcondensed]{svjour3}     % onecolumn (ditto)
\usepackage{mathrsfs}
\usepackage{amsmath}
\usepackage{amssymb}
\usepackage{amsfonts}
\usepackage{txfonts}
\newtheorem{Proposition}[theorem]{Proposition}
\numberwithin{definition}{section}
\numberwithin{theorem}{section}
\numberwithin{corollary}{section}
\numberwithin{proposition}{section}
\numberwithin{lemma}{section}
\numberwithin{equation}{section}
\numberwithin{remark}{section}
\numberwithin{example}{section}
\def\RR{\mathbb{R}}       % one dimension real space
\def\Rd{\mathbb{R}^d}     % d dimension real space
\def\vx{\boldsymbol{x}}   % x in R^d
\def\vy{\boldsymbol{y}}   % y in R^d
\def\vc{\boldsymbol{c}}
\def\NN{\mathbb{N}}       % positive integer
\def\ud{\mathrm{d}}       % integral symbol d
\def\Schwartz{\mathcal{S}} % Schwartz Space
\newcommand{\norm}[1]{\lVert#1\rVert}  % norm of R^d
\newcommand{\abs}[1]{\lvert#1\rvert}   % absolute value of real or complex number
\newcommand{\Matlab}{{\sc Matlab}}
\def\vP{\mathbf{P}} % vector distributional/differential operator
\def\vB{\mathbf{B}} % vector boundary operator
\def\veta{\boldsymbol{\eta}}

\def\vGamma{\boldsymbol{\Gamma}}

\def\vGamma{\mathbf{\Gamma}}

\def\v0{\boldsymbol{0}}
\def\DOmega{\mathscr{D}(\Omega)}           % infinity smooth functions space with compact support on Omega
\def\LlocOmega{\mathrm{L}_1^{loc}(\Omega)} % local integral functions space on Omega
\def\Cont{\mathrm{C}}

\def\Leb{\mathrm{L}}
\def\Order{\mathcal{O}}  % Order
\def\Null{\mathcal{N}ull}
\def\Hil{\mathcal{H}}
\def\Hilbert{\mathrm{H}}

\def\HpOmega{\mathrm{H}_{\mathbf{P}}^0(\Omega)}
\def\HbAOmega{\mathrm{H}_{\mathbf{B}}^{\mathscr{A}}(\Omega)}

\def\HpbAOmega{\mathrm{H}^{\mathscr{A}}_{\mathbf{P}\mathbf{B}}(\Omega)}
\def\Pm{\mathscr{P}^m_{\Omega}}
\def\Bm{\mathscr{B}^m_{\Omega}}
\def\np{n_p}
\def\nb{n_b}
\def\na{n_a}
\def\Integral{\mathcal{I}}

\def\Span{\mathrm{span}}
\def\Range{\mathcal{R}ange}
\def\Aset{\mathscr{A}}
\def\Rset{\mathscr{R}}
\def\Eset{\mathscr{E}}
\def\dualDOmega{\mathscr{D}'(\Omega)}
\begin{document}

%----------------------------------------------------------------------------------------------------------
%----------------------------------------------------------------------------------------------------------

\title{Reproducing Kernels of Sobolev Spaces via a Green Kernel Approach with Differential Operators \& Boundary Operators}

\titlerunning{Green Kernels and Reproducing Kernels with Differential and Boundary Operators}

%----------------------------------------------------------------------------------------------------------
%----------------------------------------------------------------------------------------------------------

\author{
Gregory E. Fasshauer
\and
Qi Ye
}

\institute{Gregory E. Fasshauer \at
              Department of Applied Mathematics, Illinois Institute of Technology, Chicago IL 60616 USA\\
              \email{fasshauer@iit.edu}
           \and
           Qi Ye \at
              Department of Applied Mathematics, Illinois Institute of Technology, Chicago IL 60616 USA\\
              Tel.: +1-312-451-7342\\
              Fax: +1-312-567-3135\\
              \email{qye3@iit.edu}
}

\date{}

\maketitle

%----------------------------------------------------------------------------------------------------------
%----------------------------------------------------------------------------------------------------------

\begin{abstract}
%% Text of abstract

We introduce a vector differential operator $\vP$ and a vector boundary operator $\vB$ to derive a reproducing kernel along with its associated Hilbert space which is shown to be embedded in a classical Sobolev space. This reproducing kernel is a Green kernel of differential operator $L:=\vP^{\ast T}\vP$ with homogeneous or nonhomogeneous boundary conditions given by $\vB$, where we ensure that the distributional adjoint operator $\vP^{\ast}$ of $\vP$ is well-defined in the distributional sense. We represent the inner product of the reproducing-kernel Hilbert space in terms of the operators $\vP$ and $\vB$. In addition, we find relationships for the eigenfunctions and eigenvalues of the reproducing kernel and the operators with homogeneous or nonhomogeneous boundary conditions. These eigenfunctions and eigenvalues are used to compute a series expansion of the reproducing kernel and an orthonormal basis of the reproducing-kernel Hilbert space. Our theoretical results provide perhaps a more intuitive way of understanding what kind of functions are well approximated by the reproducing kernel-based interpolant to a given multivariate data sample.

\end{abstract}

%\textbf{Mathematics Subject Classification (2000)}: Primary 41A30,
%65D05; Secondary 34B27, 41A63, 46E22, 46E35

%---------------------------------------------------------------------------------------------------------------------
%/////////////////////////////////////////////////////////////////////////////////////////////////////////////////////
%---------------------------------------------------------------------------------------------------------------------

\section{Introduction}

The reproducing-kernel Hilbert space construction associates a positive definite kernel with a Hilbert space of functions often referred to as the native space of the kernel. This construction can be used to deal with the problem of reconstructing an unknown function which lies in the reproducing-kernel Hilbert space from a given multivariate data sample (see~\cite{Fas07,Wen05}) in an ``optimal'' way. Here this optimality can be quantified in terms of the norm induced by the Hilbert space inner product. It is therefore of importance to understand these spaces (and their inner products) as well as possible since such an understanding will provide us with insight into the ``correct'' choice of kernel for any given application. Potential applications of kernel approximation methods can be found in an increasingly wider array of topics of which we mention only scattered data approximation~\cite{Buh03,DevRon10,Fas07,Sch08,Wen05}, numerical solution of partial differential equations \cite{Fas07,FlyWri09,ForPir08,Kan90,LarFor03,Sch07,Wen05}, statistical learning~\cite{BerThAg04,SteChr08,Wah92} and engineering design~\cite{ForSobKea08}. Future applications may see the combination of meshfree approximation methods and stochastic Kriging methods used within a common reproducing kernel framework to approximate the numerical solution of stochastic partial differential equations (see, e.g., \cite{FasYe10SPDE}).

However, kernel approximation methods still face quite a few difficulties and challenges. Two important questions in need of a satisfactory answer are: \emph{What kind of functions belong to
a given reproducing-kernel Hilbert space?} and \emph{Which kernel
function should we utilize for a particular application?} Our recent paper~\cite{FasYe10} establishes what kind of (full-space) Green function is a (conditionally) positive definite function and then shows how to embed its related reproducing kernel Hilbert space (or native space) into a generalized Sobolev space defined by a vector distributional operator $\vP=(P_1,\cdots,P_n,\cdots)^T$. This construction results in an arguably more intuitive interpretation of the reproducing kernel Hilbert space associated with any given kernel. In some cases these two spaces are even shown to be equivalent. Our theoretical results produce a rule that allows us to determine which Green function can be used to approximate (well) an unknown smooth function. Conversely, we can use a Green function to formulate an interpolant for a corresponding class of smooth functions. The framework discussed in our earlier paper was restricted to full-space Green functions defined on the whole space $\Rd$, i.e., without taking into consideration the effect of boundary conditions. In the present paper we will show that the Green kernel derived using boundary conditions in a regular bounded open domain $\Omega\subset\Rd$ is a reproducing kernel and that its reproducing kernel Hilbert space is embedded in a classical Sobolev space.
We begin by precisely defining what we mean in this paper by a function space being embedded in or being isomorphic to another space.
%////////////////////////////////////////////////////////////////////////////////////////////////////////////////////////
\begin{definition}[{\cite[Definition~1.25]{AdaFou03}}]\label{d:Imbedding}
We say the normed space $\Hilbert$ is \emph{embedded in} the normed
space $\Hil$ if $\Hilbert$ is a subspace of $\Hil$ and the identity
operator $I:\Hilbert\rightarrow\Hil$ is a bounded (continuous)
operator, i.e., there is a positive constant $C$ such that
$\norm{f}_{\Hil}\leq C\norm{f}_{\Hilbert}$ for each
$f\in\Hilbert\subseteq\Hil$. In particular, if $\Hil$ is also
embedded in $\Hilbert$ then we say that $\Hilbert$ and $\Hil$ are
\emph{isomorphic}, i.e., $\Hilbert\cong\Hil$.
\end{definition}
%////////////////////////////////////////////////////////////////////////////////////////////////////////////////////////
%
%////////////////////////////////////////////////////////////////////////////////////////////////////////////////////////
\begin{remark}
Here \emph{equality} of two function spaces, $\Hilbert=\Hil$, means that
$\Hilbert\subseteq\Hil$ and $\Hil\subseteq\Hilbert$ only, i.e., we do not compare their norms.
Unless specifically indicated otherwise, all functions discussed
in this article are real-valued.
\end{remark}
%////////////////////////////////////////////////////////////////////////////////////////////////////////////////////////

We now present a standard Green kernel example from the theory of partial differential equations (see \cite[Chapter~2.2]{Eva98}) to set the stage for our discussions later on.
In order to solve Poisson's equation in the $d$-dimensional ($d \geq 2$) open unit ball $\Omega=B(0,1)=\{\vx\in\Rd:\norm{\vx}_2<1\}$ with (homogeneous) Dirichlet boundary condition, one constructs the Green kernel
\[
G(\vx,\vy)=\phi(\vx-\vy)-\phi(\norm{\vx}_2\vy-\vx),\quad{}\vx,\vy\in\Omega,
\]
of the Laplace operator $L=-\Delta=-\sum_{j=1}^{d}\frac{\partial^2}{\partial x_j^2}$ subject to the given boundary condition, i.e., for each fixed $\vy\in \Omega$, we have $G(\cdot,\vy)\in\Hil^1(\Omega)$ (see Section~\ref{s:DiffAdjoint} below for the definition of the classical $\Leb_2$-based Sobolev spaces $\Hil^m(\Omega)$) and
\[
\begin{cases}
L G(\cdot,\vy)=\delta_{\vy},&\text{in }\Omega,\\
~~G(\cdot,\vy)=0,&\text{on }\partial\Omega,
\end{cases}
\]
where $\phi$ is the fundamental solution of $-\Delta$ given by
\[
\phi(\vx)=
\begin{cases}
-\frac{1}{2\pi}\log\norm{\vx}_2,&d=2,\\
\frac{\Gamma(d/2+1)}{d(d-2)\pi^{d/2}}\norm{\vx}_2^{2-d},&d\geq 3.
\end{cases}
\]
Just as in our discussion below, the Laplace operator $L=-\Delta=\vP^{\ast T}\vP=-\nabla^T\nabla$ can be computed using the gradient $\vP=(P_1,\cdots,P_d)^T=\nabla=(\frac{\partial}{\partial x_1},\cdots,\frac{\partial}{\partial x_d})^T$ and its adjoint $\vP^{\ast}=(P^{\ast}_1,\cdots,P^{\ast}_d)^T=-\nabla$. With the help of Green's formulas~\cite{Eva98} we can further check that the kernel $G$ satisfies a reproducing property with respect to the gradient-semi-inner product, i.e., for all $f\in\Cont^1_0(\overline{\Omega})$ and $\vy\in \Omega$, we have
\[
(G(\cdot,\vy),f)_{\nabla,\Omega} =\int_{\Omega}\vP G(\vx,\vy)^T \vP f(\vx)\ud\vx = \sum_{j=1}^{d}\int_{\Omega}\frac{\partial}{\partial x_j}G(\vx,\vy)\frac{\partial}{\partial x_j}f(\vx)\ud\vx
=f(\vy).
\]
However, this Green kernel $G$ is not a reproducing kernel (cf.~Definition~\ref{d:RKHS}) because $G$ is singular along its diagonal, i.e., $G(\vx,\vx)=\infty$ for each $\vx\in\Omega$.

Therefore, it is our goal to show what kind of Green kernel is a reproducing kernel while maintaining a similar concept for the reproducing property.
Our Green kernel will be associated with a differential operator $L$ with homogeneous or nonhomogeneous boundary conditions (see Definition~\ref{d:Green-Nonhomo}), and the inner product of its reproducing-kernel Hilbert space will be represented through a vector differential operator $\vP=(P_1,\cdots,P_{\np})^T$ and a vector boundary operator $\vB=(B_1,\cdots,B_{\nb})^T$, where the differential operators $P_j:\Hil^m(\Omega)\rightarrow\Leb_2(\Omega)$ and the boundary operators $B_j:\Hil^m(\Omega)\rightarrow\Leb_2(\partial\Omega)$ are bounded linear operators which are defined and discussed
in Section~\ref{s:DiffBound}.

Because the Dirac delta function $\delta_{\vy}$ is a tempered distribution in the dual space $\dualDOmega$ of the test function space $\DOmega$ (see Section~\ref{s:DiffAdjoint}) we shall extend the differential operators and their adjoint operators to distributional operators from $\dualDOmega$ into $\dualDOmega$. Thus the differential operator $L$ can be represented by the vector differential operator $\vP$ and its distributional adjoint operator $\vP^{\ast}$ via the formula $L=\vP^{\ast T}\vP=\sum_{j=1}^{\nb}P^{\ast}_jP_j$. In this article, a differential operator $P$, its distributional adjoint operator $P^{\ast}$ and a boundary operator $B$ are assumed to be linear with non-constant coefficients, i.e.,
\[
P=\sum_{\abs{\alpha}\leq m}\rho_{\alpha}\circ
D^{\alpha},\quad
P^{\ast}=\sum_{\abs{\alpha}\leq
m}(-1)^{\abs{\alpha}}D^{\alpha}\circ\rho_{\alpha},\quad
B=\sum_{\abs{\beta}\leq
m-1}b_{\beta}\circ D^{\beta}|_{\partial\Omega},
\]
where $\rho_{\alpha}\in\Cont^{\infty}(\overline{\Omega})$, $b_{\beta}\in\Cont(\partial\Omega)$ and $\alpha,\beta\in\NN_0^d$ (see Definition~\ref{d:diff} and~\ref{d:bound}).

Based on this construction we can establish a direct connection between Green kernels and reproducing kernels. We are also able to show how to use the differential operator $\vP$ and boundary operator $\vB$ to set up reproducing kernel Hilbert spaces which are embedded in classical Sobolev spaces (see Section~\ref{s:RKHS_Sob}).
For example, Theorems~\ref{t:HpbA}, Corollary~\ref{c:HpbAOmega-Null} and Theorem~\ref{t:RKHS-HpbAOmega} allow us to arrive at a theorem such as
%////////////////////////////////////////////////////////////////////////////////////////////////////////////////////////
\begin{theorem}\label{t:finite-dim-RKHS}
Let $\Omega\subset\Rd$ be a regular bounded open domain and introduce the vector differential operator $\vP=(P_1,\cdots,P_{\np})^T\in\Pm$ and vector boundary operator $\vB=(B_1,\cdots,B_{\nb})^T\in\Bm$, where $m>d/2$ and $m\in\NN$. Suppose that there is a Green kernel $G$ of $L=\vP^{\ast T}\vP=\sum_{j=1}^{\np}P^{\ast}_jP_j$ with homogeneous boundary conditions given by $\vB$, i.e., for each fixed $\vy\in\Omega$, we have $G(\cdot,\vy)\in\Hil^m(\Omega)$ and
\[
\begin{cases}
LG(\cdot,\vy)=\delta_{\vy},&\text{in }\Omega,\\
\vB G(\cdot,\vy)=\v0,&\text{on }\partial\Omega.
\end{cases}
\]
If the null space $\Null(\vP):=\{f\in\Hil^m(\Omega):\vP f=\v0\}$ is a finite-dimensional space, then the direct sum space
\[
\HpbAOmega=\HpOmega\oplus\HbAOmega
=\left\{f=f_P+f_B:\vB f_P=\v0,~\vP f_B=\v0,\text{ where }f_P,f_B\in\Hil^m(\Omega)\right\}
\]
equipped with the inner product
\[
(f,g)_{\HpbAOmega}=\sum_{j=1}^{\np}\int_{\Omega}P_jf(\vx)P_jg(\vx)\ud\vx
+\sum_{j=1}^{\nb}\int_{\partial\Omega}B_jf(\vx)B_jg(\vx)\ud S(\vx),\quad f,g\in\HpbAOmega,
\]
is a reproducing-kernel Hilbert space whose reproducing kernel is a Green kernel $K$ of $L$ with boundary conditions given by $\vB$ and
$\{\vGamma(\cdot,\vy):\vy\in\Omega\}\subseteq\otimes_{j=1}^{\nb}\Leb_2(\partial\Omega)$, i.e., for each fixed $\vy\in\Omega$, we have $K(\cdot,\vy)\in\Hil^m(\Omega)$ and
\[
\begin{cases}
LK(\cdot,\vy)=\delta_{\vy},&\text{in }\Omega,\\
\vB K(\cdot,\vy)=\vGamma(\cdot,\vy),&\text{on }\partial\Omega,
\end{cases}
\]
where the boundary conditions also satisfy $\{\vGamma(\vx,\cdot):\vx\in\partial\Omega\}\subseteq\otimes_{j=1}^{\nb}\Null(\vP)$.
Moreover, the reproducing-kernel Hilbert space $\HpbAOmega$ is embedded in the Sobolev space $\Hil^m(\Omega)$ and the reproducing kernel $K$ can be written in the explicit form
\[
K(\vx,\vy)=G(\vx,\vy)+\sum_{k=1}^{\na}\psi_k(\vx)\psi_k(\vy),\quad{}\vx,\vy\in\Omega,
\]
where $\{\psi_k\}_{k=1}^{\na}$ is an orthonormal basis of $\Null(\vP)$ with respect to the $\vB$-semi-inner product.
(Here the classes $\Pm$ and $\Bm$ are defined in Section~\ref{s:DiffBound}.)
\end{theorem}
%////////////////////////////////////////////////////////////////////////////////////////////////////////////////////////

Theorem~\ref{t:finite-dim-RKHS} shows that the vector differential operator $\vP$ and vector boundary operator $\vB$ enable us to verify the reproducing property of the reproducing-kernel Hilbert space. This allows us to show that the Green kernel $K$ becomes a reproducing kernel even with nonhomogeneous boundary conditions, not just for the case of homogeneous boundary conditions.
If $\Null(\vP)\equiv\{0\}$ then $K=G$ has homogeneous boundary conditions
which implies that the reproducing property depends on $\vP$ without having to resort to $\vB$ -- just as we had above for the case of the Poisson Green kernel. We can now reconsider the question of why the Poisson Green kernel above is \emph{not} a reproducing kernel. Essentially this happens because $m=1\leq d/2$ so that the Sobolev embedding theory does not apply. On the other hand, Remark~\ref{r:Green-non} gives us a counter example demonstrating that the Green kernel may not be a reproducing kernel even if it is uniformly continuous in the whole domain.

In Section~\ref{s:RKHS_Sob} we also consider the solution of eigenvalue problems via the method presented in~\cite{AtkHan09},
where the authors discuss how to find the eigenfunctions and eigenvalues
of elliptic partial differential equations of order $2$ with
Dirichlet or Neumann boundary conditions. This will enable us to see the relationships between the eigenfunctions and eigenvalues of Green kernels and those of differential operators $L$
with homogeneous or nonhomogeneous boundary
conditions given by $\vB$. Propositions~\ref{p:W-eig-L-B-eig} and \ref{p:W-eig-L-B-eig-non} allow us to transfer eigenfunctions and eigenvalues from Green kernels to $L$ and vice versa. We also use these eigenfunctions and eigenvalues to obtain the orthonormal basis of the reproducing-kernel
Hilbert space and the explicit expansion of the Green kernel as, e.g., stated in Proposition~\ref{p:L-B-eig-W-eig} and \ref{p:L-B-eig-W-eig-non}.

In Section~\ref{s:exa}, we demonstrate that many well-known reproducing
kernels are also Green kernels. Examples include the min kernel and the univariate Sobolev spline kernel.
We also construct other reproducing kernels that can be used in scattered data interpolation such as a modification of the thin-plate spline.

In this article we limit our discussion of nonhomogeneous boundary conditions to those that are determined by a finite bases. However, all the theoretical results presented here can be extended to much more general nonhomogeneous boundary conditions constructed using a countable basis (see the Ph.D. thesis~\cite{Ye12} of the second author). Such Green kernels $K$ can be seen as a reproducing kernel for the interpolation of multivariate scattered data obtained from an unknown function $f\in\Hil^m(\Omega)$ at data sites $X=\{\vx_j\}_{j=1}^N\subset\Omega$. In a similar fashion as described in \cite{Fas07,SteChr08,Wen05}, we further obtain error bounds and optimal recovery properties for the interpolant $s_{f,X}=\sum_{j=1}^{N}c_jK(\cdot,\vx_j)$ which satisfies the interpolation conditions $s_{f,X}(\vx_j)=f(\vx_j)$ for each $j=1,\cdots,N$.

%---------------------------------------------------------------------------------------------------------------------
%/////////////////////////////////////////////////////////////////////////////////////////////////////////////////////
%---------------------------------------------------------------------------------------------------------------------

\section{Positive Definite Kernels and Reproducing-Kernel Hilbert
Space}\label{s:PDK-RKHS}

We now provide a very brief summary of reproducing kernel Hilbert spaces. Much more background information can be found in, e.g., \cite{Wen05}.

%////////////////////////////////////////////////////////////////////////////////////////////////////////////////////////
\begin{definition}[{\cite[Definition~6.24]{Wen05}}]\label{d:PDK}
Let $\Omega\subseteq\Rd$. A symmetric kernel
$K:\Omega\times\Omega\rightarrow\RR$ is called \emph{positive
definite} if, for all $N\in\NN$, pairwise distinct points
$X:=\{\vx_1,\ldots,\vx_N\}\subset\Omega$, and
$\vc:=(c_1,\ldots,c_N)^T\in\RR^N\setminus\{0\}$ the quadratic form
\[
\sum_{j=1}^{N}\sum_{k=1}^Nc_jc_kK(\vx_j,\vx_k)>0.
\]
If the quadratic form is only nonnegative, then the kernel $K$ is said to be positive semi-definite.
\end{definition}
%////////////////////////////////////////////////////////////////////////////////////////////////////////////////////////

%////////////////////////////////////////////////////////////////////////////////////////////////////////////////////////
\begin{definition}[{\cite[Definition~10.1]{Wen05}}]\label{d:RKHS}
Let $\Omega\subseteq\Rd$ and $\Hilbert(\Omega)$ be a real Hilbert
space of functions $f:\Omega\rightarrow\RR$. $\mathrm{H}(\Omega)$ is
called a \emph{reproducing-kernel Hilbert space} with a
\emph{reproducing kernel} $K:\Omega\times\Omega\rightarrow\RR$ if
\[
\begin{split}
&(i)~K(\cdot,\vy)\in\Hilbert(\Omega)\text{ and }
(ii)~f(\vy)=(K(\cdot,\vy),f)_{\Hilbert(\Omega)},\quad \text{for all
}f\in\Hilbert(\Omega)\text{ and each }\vy\in\Omega.
\end{split}
\]
\end{definition}
%////////////////////////////////////////////////////////////////////////////////////////////////////////////////////////

In order to formulate the following proposition which we will later use to verify some of our results on eigenfunctions and eigenvalues of a Green kernel we first consider a kernel $K\in\Leb_2(\Omega\times\Omega)$ and define an \emph{integral operator}
$\Integral_{K,\Omega}:\Leb_2(\Omega)\rightarrow\Leb_2(\Omega)$
via
\begin{equation}\label{e:IntOp}
(\Integral_{K,\Omega}f)(\vy):=\int_{\Omega}K(\vx,\vy)f(\vx)\ud\vx,\quad{}f\in\Leb_2(\Omega)\text{
and }\vy\in\Omega.
\end{equation}

%////////////////////////////////////////////////////////////////////////////////////////////////////////////////////////
\begin{Proposition}[{\cite[Proposition~10.28]{Wen05}}]\label{p:RKHS-Ltwo}
Suppose that the reproducing kernel $K\in\Leb_2(\Omega\times\Omega)$ is a symmetric positive definite kernel on the compact set $\Omega\subseteq\Rd$.
Then the integral operator $\Integral_{K,\Omega}$ maps $\Leb_2(\Omega)$ continuously into the reproducing-kernel Hilbert space $\Hilbert(\Omega)$ whose reproducing kernel is $K$. The operator $\Integral_{K,\Omega}$ is the adjoint of the embedding operator of the reproducing-kernel Hilbert space $\Hilbert(\Omega)$ into $\Leb_2(\Omega)$, i.e., it satisfies
\[
\int_{\Omega}f(\vx)g(\vx)\ud\vx=(f,\Integral_{K,\Omega}g)_{\Hilbert(\Omega)},\quad{}
f\in\Hilbert(\Omega)\text{ and }g\in\Leb_2(\Omega).
\]
Moreover, $\Range(\Integral_{K,\Omega})=\{\Integral_{K,\Omega}g:g\in\Leb_2(\Omega)\}$ is dense in $\Hilbert(\Omega)$ with respect to the $\Hilbert(\Omega)$-norm.
\end{Proposition}
%////////////////////////////////////////////////////////////////////////////////////////////////////////////////////////

%---------------------------------------------------------------------------------------------------------------------
%/////////////////////////////////////////////////////////////////////////////////////////////////////////////////////
%---------------------------------------------------------------------------------------------------------------------

\section{Differential Operators and Boundary Operators}\label{s:DiffBound}

%---------------------------------------------------------------------------------------------------------------------
%---------------------------------------------------------------------------------------------------------------------

\subsection{Differential Operators and Distributional Adjoint
Operators}\label{s:DiffAdjoint}

Our following proofs will rely on a number of basic concepts and techniques from the Schwartz theory of distributions (see \cite[Chapter~1.5]{AdaFou03} and \cite[Chapter~1~and~2]{Hor0405}). Of special importance is the notion of a distributional derivative of an integrable function. Distributional derivatives are extensions of the standard partial derivatives
\[
D^{\alpha}:=\prod_{k=1}^{d}\frac{\partial^{\alpha_k}}{\partial
x_k^{\alpha_k}}, \quad{}\abs{\alpha}:=\sum_{k=1}^d\alpha_k,\quad{}
\alpha:=\left(\alpha_1,\cdots,\alpha_d\right)\in\NN_0^d.
\]

Let $\Omega\subset\Rd$ be an open bounded domain (connected subset). We first introduce a test function space $\Cont_0^{\infty}(\Omega)$ which consists of all those functions in $\Cont^{\infty}(\Omega)$ having compact support in $\Omega$. \cite[Chapter~1.5]{AdaFou03} states that the test function space $\Cont_0^{\infty}(\Omega)$ can be given a locally convex topology and thereby becomes a topological vector space called $\DOmega$. Note, however, that $\DOmega$ is not a normable space.

Its dual space $\dualDOmega$ (the space of continuous functionals on $\DOmega$) is referred to as the space of tempered distributions. According to \cite[Chapter~2.1]{Hor0405}, a distribution $T\in\dualDOmega$ is a linear form on $\DOmega$ such that for every compact set $\Lambda\subset\Omega$ there exist a positive constant $C$ and a nonnegative integer $n\in\NN_0$ such that
\[
T(\gamma)\leq C\sum_{\abs{\alpha}\leq n}\sup_{\vx\in\Lambda}\abs{D^{\alpha}\gamma(\vx)},\quad
\text{for each }\gamma\in\Cont_0^{\infty}(\Lambda)\subset\DOmega.
\]
For example, the Dirac delta
function (Dirac delta distribution) $\delta_{\vy}$ concentrated at the
point $\vy\in\Omega$ is an element of $\dualDOmega$, i.e.,
$\langle \delta_{\vy},\gamma\rangle_{\Omega} =\gamma(\vy)$ for each
$\gamma\in\DOmega$.
%%%
Our later proofs will make frequent use of the following two bilinear forms. We define a \emph{dual bilinear form}
\[
\langle T,\gamma\rangle_{\Omega}:=T(\gamma),\quad\text{for each
}T\in\dualDOmega \text{ and }\gamma\in\DOmega,
\]
and the usual \emph{integral bilinear form}
\[
(f,g)_{\Omega}:=\int_{\Omega}f(\vx)g(\vx)\ud\vx,\quad
\text{where $fg$ is integrable on $\Omega$}.
\]
\cite[Chapter~1.5]{Hor0405} shows that for each locally integrable function $f\in\LlocOmega$ there exists a unique tempered distribution
$T_f\in\dualDOmega$ that links these two bilinear forms by the Riesz representation theorem, i.e.,
\begin{equation}\label{e:DistNotation}
\langle T_f,\gamma\rangle_{\Omega}
=(f,\gamma)_{\Omega},\quad\text{for each } \gamma\in\DOmega.
\end{equation}
Thus $f\in\LlocOmega$ can be viewed as an element of $\dualDOmega$ and $T_f$ is frequently identified with $f$. This means that $\LlocOmega\subset\dualDOmega$.

Next we extend the standard derivative $D^{\alpha}$ to the notion of a distributional derivative $P^{\alpha}:\dualDOmega\rightarrow\dualDOmega$. This distributional derivative is well defined by
\[
\langle P^{\alpha}T, \gamma \rangle_{\Omega}:=(-1)^{\alpha}\langle T, D^{\alpha}\gamma \rangle,\quad
\text{for each }T\in\dualDOmega\text{ and }\gamma\in\DOmega,
\]
because $D^{\alpha}$ is continuous from $\DOmega$ into $\DOmega$ (see \cite[Definition~3.1.1]{Hor0405}). For convenience $P^{\alpha}$ is also written as $D^{\alpha}$.

Using this notion of distributional derivatives the real \emph{classical $\Leb_2$-based Sobolev space} $\Hil^m(\Omega)$ is
defined by
\[
\Hil^m(\Omega):=\left\{f\in\LlocOmega:D^{\alpha}f\in\Leb_2(\Omega),~\abs{\alpha}\leq
m,~\alpha\in\NN_0^d\right\},\quad{}m\in\NN_0,
\]
equipped with the natural inner product
\[
(f,g)_{m,\Omega}:=\sum_{\abs{\alpha}\leq
m}\int_{\Omega}D^{\alpha}f(\vx)D^{\alpha}g(\vx)\ud\vx,\quad f,g\in\Hil^m(\Omega).
\]
Moreover, the completion of $\Cont_0^m(\Omega)$ with respect to
the $\Hil^m(\Omega)$-norm is denoted by $\Hil^m_0(\Omega)$, i.e.,
$\Hil^m_0(\Omega)$ is the closure of $\Cont_0^{\infty}(\Omega)$ in
$\Hil^m(\Omega)$ as in~\cite{AdaFou03}.

In the literature (see, e.g., \cite{Hor0405}) one also often finds differential operators written in the form $p(\cdot,D)\gamma=\sum_{\abs{\alpha}\leq m}\rho_{\alpha}D^{\alpha}\gamma$, where
$p(\vx,\vy):=\sum_{\abs{\alpha}\leq m}\rho_{\alpha}(\vx)\vy^{\alpha}$ is a polynomial in $\vy\in\Rd$ and $\rho_{\alpha}\in\Cont^{\infty}(\overline{\Omega})$ (uniformly smooth functions space). The formal adjoint operator can be represented as
$p^{\ast}(\cdot,D)\gamma=\sum_{\abs{\alpha}\leq m}(-1)^{\abs{\alpha}}D^{\alpha}(\rho_{\alpha}\gamma)$. If $\rho\in\Cont^{\infty}(\overline{\Omega})$ then it can be seen as a distributional operator $P_{\rho}:\dualDOmega\rightarrow\dualDOmega$, i.e.,
\[
\langle P_{\rho}T,\gamma \rangle:=\langle T,\rho\gamma\rangle,\quad
\text{for each }T\in\dualDOmega\text{ and }\gamma\in\DOmega,
\]
because $\gamma\mapsto\rho\gamma$ is continuous from $\DOmega$ into $\DOmega$ (see \cite[Definition~3.1.1]{Hor0405}). Here we identify $P_{\rho}$ with $\rho$. Then this differential operator $p(\cdot,D)$ and its adjoint operator $p^{\ast}(\cdot,D):\DOmega\rightarrow\DOmega$ can be extended to distributional operators $P,P^{\ast}:\dualDOmega\rightarrow\dualDOmega$ similar as the distributional derivatives. To avoid any confusion with the symbols we will write $P_1P_2=\rho\circ D^{\alpha}$ and $P_2P_1=D^{\alpha}\circ\rho$ where $P_1=\rho$ and $P_2=D^{\alpha}$. This means that
\[
\rho\circ
D^{\alpha}\gamma=\rho\left(D^{\alpha}\gamma\right),
\quad D^{\alpha}\circ\rho\gamma=(-1)^{\abs{\alpha}}D^{\alpha}\left(\rho\gamma\right),
\quad \gamma\in\DOmega.
\]

%////////////////////////////////////////////////////////////////////////////////////////////////////////////////////////
\begin{definition}\label{d:diff}
A \emph{differential operator} (with non-constant coefficients)
$P:\dualDOmega\rightarrow\dualDOmega$ is defined by
\[
P=\sum_{\abs{\alpha}\leq m}\rho_{\alpha}\circ
D^{\alpha},\quad{}\text{where
}\rho_{\alpha}\in\Cont^{\infty}(\overline{\Omega})\text{ and
}\alpha\in\NN_0^d,~m\in\NN_0.
\]
Its distributional adjoint operator
$P^{\ast}:\dualDOmega\rightarrow\dualDOmega$ is well-defined by
\[
P^{\ast}=\sum_{\abs{\alpha}\leq
m}(-1)^{\abs{\alpha}}D^{\alpha}\circ\rho_{\alpha}.
\]
We further denote its \emph{order} by
\[
\Order(P):=\max\left\{\abs{\alpha}:\rho_{\alpha}\not\equiv0,~\abs{\alpha}\leq
m,~\alpha\in\NN_0^d\right\}.
\]
A \emph{vector differential operator} $\vP:=(P_1,\cdots,P_{\np})^T$
is constructed using a finite number of differential operators
$P_1,\cdots,P_{\np}$ and its \emph{order}
$\Order(\vP):=\max\{\Order(P_1),\cdots,\Order(P_{\np})\}$.
\end{definition}
%////////////////////////////////////////////////////////////////////////////////////////////////////////////////////////

After replacing the test function space
$\Schwartz$ (metric space of rapidly decreasing functions in $\Cont^{\infty}(\Rd)$) and tempered distribution space $\Schwartz'$ (dual space of $\Schwartz$) in paper~\cite{FasYe10}, the differential operator $P$ and its distributional adjoint operator $P^{\ast}$ have the same properties as \cite[Definition~4.1]{FasYe10}, i.e.,
$P|_{\DOmega}$ and $P^{\ast}|_{\DOmega}$ are
continuous operators from $\DOmega$ into $\DOmega$ and
\[
\langle PT,\gamma\rangle_{\Omega}=\langle T,P^{\ast}\gamma\rangle_{\Omega}~\text{ and
}~\langle P^{\ast}T,\gamma\rangle_{\Omega}=\langle T,P\gamma\rangle_{\Omega},
\text{ for each }T\in\dualDOmega\text{ and }\gamma\in\DOmega.
\]
Since $\overline{\Omega}$ is compact and $\Cont^{\infty}(\overline{\Omega})\subset\Leb_2(\Omega)$,
the differential operator $P$ of order
$\Order(P)=m$ is a bounded linear operator from $\Hil^m(\Omega)$ into
$\Leb_2(\Omega)$. Its distributional adjoint operator $P^{\ast}:\Hil^m(\Omega)\rightarrow\Leb_2(\Omega)$ is also bounded.
So we can further use a vector differential operator
$\vP:=(P_1,\cdots,P_{\np})^T$ of order $m$ to define a
\emph{$\vP$-semi-inner product} on $\Hil^m(\Omega)$ via the form
\[
(f,g)_{\vP,\Omega}=\sum_{j=1}^{\np}(P_jf,P_jg)_{\Omega},\quad{}
f,g\in\Hil^m(\Omega).
\]
%////////////////////////////////////////////////////////////////////////////////////////////////////////////////////////
\begin{remark}
Our distributional adjoint operator differs from the classical adjoint
operator of a bounded operator defined in Hilbert space or Banach
space. Our operator is defined in the dual space of $\DOmega$ and it
may not be continuous if the dual of $\DOmega$ is defined by its
natural topology. But the differential operator and its distributional adjoint operator are continuous when $\dualDOmega$ is given the
weak-star topology as the dual of $\DOmega$, i.e., $T_k\rightarrow T$ in $\dualDOmega$ if and only if
$\langle T_k,\gamma \rangle_{\Omega}\rightarrow\langle T,\gamma \rangle_{\Omega}$ for every $\gamma\in\DOmega$ where $\{T_k,T\}_{k=1}^{\infty}\subset\dualDOmega$.
\end{remark}
%////////////////////////////////////////////////////////////////////////////////////////////////////////////////////////
%

When $\vP=\nabla=(\frac{\partial}{\partial
x_1},\cdots,\frac{\partial}{\partial x_d})^T$ the
$\vP$-semi-inner product is the same as the gradient-semi-inner product
on the Sobolev space $\Hil^1(\Omega)$. The Poincar\'{e}
inequality~\cite[Theorem~12.77]{HunNac05} states that the gradient-semi-norm is
equivalent to the $\Hil^1(\Omega)$-norm on the space
$\Hil^1_0(\Omega)$, i.e., there are two positive constants $C_1$ and
$C_2$ such that
\[
C_1\norm{f}_{1,\Omega}\leq\abs{f}_{\nabla,\Omega}\leq
C_2\norm{f}_{1,\Omega}, \quad{}f\in\Hil^1_0(\Omega).
\]
In order to prove a generalized Poincar\'e (Sobolev) inequality for the Sobolev spaces $\Hil^m(\Omega)$ we need to set up a special class of vector differential operators.

%///////////////////////////////////////////////////////////////////////////////////////////////////////////////////////
\begin{definition}\label{d:P-semi-inner-product}
$\Pm$ is defined to be a collection of vector differential operators $\vP=(P_1,\cdots,P_{\np})^T$
of order $m\in\NN$ which satisfy the requirements that for each fixed $\abs{\alpha}=m$ and
$\alpha\in\NN_0^d$, there is an element $P_{j(\alpha)}\in\{P_{j}\}_{j=1}^{\np}$ such
that
\[
P_{j(\alpha)}^{\ast}P_{j(\alpha)}=(-1)^{\abs{\alpha}}D^{\alpha}\circ\rho_{\alpha}^2\circ
D^{\alpha}+\sum_{i=1}^{n(\alpha)}Q_{\alpha,i}^{\ast}Q_{\alpha,i},\quad{}1\leq
j(\alpha)\leq\np,\quad{}n(\alpha)\in\NN_0,
\]
where $\rho_{\alpha}\in\Cont^{\infty}(\overline{\Omega})$ is
positive in the whole domain $\overline{\Omega}$ and $Q_{\alpha,i}$, $Q_{\alpha,i}^{\ast}$, $i=1,\cdots,n(\alpha)$, are differential
operators and their distributional adjoint operators.
\end{definition}
%///////////////////////////////////////////////////////////////////////////////////////////////////////////////////////

Let's consider an example. If $d=2$, then both vector differential operators $\vP_1:=(P_{11},P_{12},P_{13})^T=(\frac{\partial^2}{\partial
x_1^2},\sqrt{2}\frac{\partial^2}{\partial x_1\partial
x_2},\frac{\partial^2}{\partial x_2^2})^T$ and $\vP_2:=P_{21}=\Delta$ belong
to $\mathscr{P}_{\Omega}^2$ because
\[
\begin{cases}
P_{11}^{\ast}P_{11}=D^{\alpha}\circ1\circ D^{\alpha},
&\text{where }\alpha=(2,0),\\
P_{12}^{\ast}P_{12}=D^{\alpha}\circ2\circ D^{\alpha},
&\text{where }\alpha=(1,1),\\
P_{13}^{\ast}P_{13}=D^{\alpha}\circ1\circ D^{\alpha},
&\text{where }\alpha=(0,2),
\end{cases}
\]
and (using the definitions of $P_{1j}$ just made)
\[
\begin{cases}
P_{21}^{\ast}P_{21}=D^{(2,0)}\circ1\circ D^{(2,0)}+
P_{12}^{\ast}P_{12}+P_{13}^{\ast}P_{13},\\
P_{21}^{\ast}P_{21}=D^{(1,1)}\circ2\circ D^{(1,1)}+
P_{11}^{\ast}P_{11}+P_{13}^{\ast}P_{13},\\
P_{21}^{\ast}P_{21}=D^{(0,2)}\circ1\circ D^{(0,2)}+
P_{11}^{\ast}P_{11}+P_{12}^{\ast}P_{12}.
\end{cases}
\]
Therefore we can verify that $\vP_1^{\ast T}\vP_1=\sum_{j=1}^{3}P_{1j}^{\ast}P_{1j}=\vP_2^{\ast T}\vP_2=P_{21}^{\ast}P_{21}=\Delta^2$. However, the null spaces of $\vP_1$ and $\vP_2$ are different, in fact $\Null(\vP_1)\varsubsetneqq\Null(\vP_2)$.

The following lemma extends the Poincar\'e inequality from the usual gradient semi-norm to more general $\vP$-semi norms and higher-order Sobolev norms. Since we could not find it anywhere in the literature we provide a proof.

%///////////////////////////////////////////////////////////////////////////////////////////////////////////////////////
\begin{lemma}\label{l:P-semi-norm}
If $\vP\in\Pm$ then there exist two positive constants $C_1$ and $C_2$
such that
\begin{equation}\label{e:P-semi-norm}
C_1\norm{f}_{m,\Omega}\leq\abs{f}_{\vP,\Omega}\leq
C_2\norm{f}_{m,\Omega}, \quad{}f\in\Hil^m_0(\Omega).
\end{equation}
\end{lemma}
%///////////////////////////////////////////////////////////////////////////////////////////////////////////////////////
\begin{proof}
By the method of induction, we can easily check that the second inequality in~(\ref{e:P-semi-norm}) is true. We now verify the first inequality in~(\ref{e:P-semi-norm}). Fixing any
$f\in\Hil_0^m(\Omega)$, there is a sequence
$\{\gamma_k\}_{k=1}^{\infty}\subset\DOmega$ so that
$\norm{\gamma_k-f}_{m,\Omega}\rightarrow0$ when
$k\rightarrow\infty$. Because of $\vP\in\Pm$, for each fixed
$\abs{\alpha}=m$ and $\alpha\in\NN_0^d$, there is an element
$P_{j(\alpha)}$ of $\vP$ such that
\[
\begin{split}
\norm{P_{j(\alpha)}f}_{\Omega}^2&=(P_{j(\alpha)}f,P_{j(\alpha)}f)_{\Omega}
=\lim_{k\rightarrow\infty}(P_{j(\alpha)}\gamma_k,P_{j(\alpha)}\gamma_k)_{\Omega}
=\lim_{k\rightarrow\infty}(P_{j(\alpha)}^{\ast}P_{j(\alpha)}\gamma_k,\gamma_k)_{\Omega}\\
&=\lim_{k\rightarrow\infty}((-1)^{\abs{\alpha}}D^{\alpha}\circ\rho_{\alpha}^2\circ
D^{\alpha}\gamma_k,\gamma_k)_{\Omega}
+\lim_{k\rightarrow\infty}\sum_{i=1}^{n(\alpha)}(Q_{\alpha,i}^{\ast}Q_{\alpha,i}\gamma_k,\gamma_k)_{\Omega}\\
&=\lim_{k\rightarrow\infty}(\rho_{\alpha}\circ
D^{\alpha}\gamma_k,\rho_{\alpha}\circ
D^{\alpha}\gamma_k)_{\Omega}
+\lim_{k\rightarrow\infty}\sum_{i=1}^{n(\alpha)}(Q_{\alpha,i}\gamma_k,Q_{\alpha,i}\gamma_k)_{\Omega}\\
&=(\rho_{\alpha}\circ D^{\alpha}f,\rho_{\alpha}\circ
D^{\alpha}f)_{\Omega} +\sum_{i=1}^{n(\alpha)}(Q_{\alpha,i}f,Q_{\alpha,i}f)_{\Omega}\geq\norm{\rho_{\alpha}D^{\alpha}f}_{\Omega}^2\\
&\geq
\min_{\vx\in\overline{\Omega}}\abs{\rho_{\alpha}(\vx)}^2\norm{D^{\alpha}f}_{\Omega}^2.
\end{split}
\]
Since the uniformly continuous function $\rho_{\alpha}$ is
positive in the compact subset $\overline{\Omega}$, we have
$\min_{\vx\in\overline{\Omega}}\abs{\rho_{\alpha}(\vx)}>0$. Therefore,
\[
C_P^2\sum_{\abs{\alpha}=m}\norm{D^{\alpha}f}_{\Omega}^2\leq\abs{f}^2_{\vP,\Omega},
\]
where $C_P^2:=\np^{-d}\min\left\{\abs{\rho_{\alpha}(\vx)}^2:\vx\in\overline{\Omega},~\abs{\alpha}=m,~\alpha\in\NN_0^d\right\}>0$.
According to the Sobolev inequality~\cite[Theorem~4.31]{AdaFou03}, there exists a
positive constant $C_D$ such that
\[
C_D^2\norm{f}_{m,\Omega}^2\leq\sum_{\abs{\alpha}=m}\norm{D^{\alpha}f}_{\Omega}^2,\quad{}f\in\Hil^m_0(\Omega).
\]
By choosing $C_1:=C_PC_D>0$ we complete the proof.

\qed
\end{proof}

%/////////////////////////////////////////////////////////////////////////////////////////////////////////////////////

%---------------------------------------------------------------------------------------------------------------------
%---------------------------------------------------------------------------------------------------------------------

\subsection{Boundary Operators}\label{s:Bound}

In this section we wish to define boundary operators on the Sobolev spaces
$\Hil^m(\Omega)$, $m\in\NN$. Since these boundary operators can not
be set up in an arbitrary bounded open domain, we will assume that $\Omega\subset\Rd$ is a
\emph{regular} bounded open domain (connected subset), e.g., it should satisfy a strong local
Lipschitz condition or a uniform cone condition (see~\cite[Chapter~4.1]{AdaFou03} and~\cite[Chapter~12.10]{HunNac05}).
This means that $\Omega$ has a regular boundary trace $\partial\Omega$. Moreover $\partial\Omega$ is closed and bounded
which implies that $\partial\Omega$ is compact because the domain $\Omega$ is open and bounded.

We begin by defining special $\Leb_2$ spaces restricted to the boundary trace $\partial\Omega$ as
\[
\Leb_2(\partial\Omega):=\{f:\partial\Omega\rightarrow\RR:\norm{f}_{\partial\Omega}<\infty\}
\]
together with an inner product given by
\[
(f,g)_{\partial\Omega}:=\int_{\partial\Omega}f(\vx)g(\vx)\ud
S(\vx),\quad{}f,g\in\Leb_2(\partial\Omega).
\]
Here $\int_{\partial\Omega}f(\vx)\ud S(\vx)$ implies that $f$ is integrable on the boundary trace $\partial\Omega$
and $\ud S$ is the surface area element whenever $d\geq2$.
In the special case $d=1$ we interpret the restricted space as
\[
\Leb_2(\partial\Omega):=\left\{f:\partial\Omega=\{a,b\}\rightarrow\RR\right\},
\]
and its inner product as
\[
(f,g)_{\partial\Omega}=f(a)g(a)+f(b)g(b),\quad{}f,g\in\Leb_2(\partial\Omega),
\]
because the measure at the endpoints is defined as $S(a)=S(b)=1$.

The crucial ingredient that allows us to deal with boundary conditions will be a boundary trace mapping which restricts the derivative of an $\Hil^m(\Omega)$ function to the boundary trace $\partial\Omega$.
More precisely, for any fixed $\abs{\beta}\leq m-1$, $\beta\in\NN_0^d$, we will define the \emph{boundary trace mapping of the $\beta^{\text{th}}$ derivative $D^{\beta}$} and denote it by $D^{\beta}|_{\partial\Omega}$.
We will now show that the operator $D^{\beta}|_{\partial\Omega}$ is a well-defined bounded linear operator from
$\Hil^m(\Omega)$ into $\Leb_2(\partial\Omega)$.

When $d=1$ we have $\Omega:=(a,b)$ and $\partial\Omega:=\left\{a,b\right\}$ with $-\infty<a<b<+\infty$.
According to the Sobolev embedding theorem (Rellich-Kondrachov theorem)~\cite[Theorem~6.3]{AdaFou03}, $\Hil^m(a,b)$ is embedded in $\Cont^{m-1}([a,b])$. In this special case the boundary trace mapping of the $\beta^{\text{th}}$ derivative $D^{\beta}$, $D^{\beta}|_{\partial\Omega}:\Hil^m(a,b)\rightarrow\Leb_2(\left\{a,b\right\})$, is well-defined on
$\Hil^m(a,b)$ via
\[
(D^{\beta}|_{\left\{a,b\right\}}f)(x)=D^{\beta}f(x),\quad{}f\in\Hil^m(a,b)\text{
and }x\in\{a,b\}.
\]

In the case $d\geq2$ a linear operator $D^{\beta}|_{\partial\Omega}:\Cont^m(\overline{\Omega})\rightarrow\Cont(\partial\Omega)$
is well-defined by
\[
D^{\beta}|_{\partial\Omega}f:=D^{\beta}f|_{\partial\Omega},\quad{}f\in\Cont^m(\overline{\Omega}).
\]
According to the boundary trace embedding theorem~(\cite[Theorem~5.36]{AdaFou03} and \cite[Theorem~12.76]{HunNac05}) there
is a constant $C_{\beta}>0$ such that
\[
\norm{D^{\beta}f}_{\partial\Omega}\leq
C_{\beta}\norm{D^{\beta}f}_{1,\Omega}\leq
C_{\beta}\norm{f}_{m,\Omega}, \quad{}f\in\Cont^m(\overline{\Omega}),
\]
which shows that $D^{\beta}|_{\partial\Omega}$ is also a bounded
operator from $\Cont^m(\overline{\Omega})\subset\Hil^m(\Omega)$ into
$\Cont(\partial\Omega)\subset\Leb_2(\partial\Omega)$. Since $\Omega$
is assumed to be regular, $\Cont^m(\overline{\Omega})$ is dense in
$\Hil^m(\Omega)$ with respect to the $\Hil^m(\Omega)$-norm by the density theorem for Sobolev
spaces~\cite[Theorem~12.69]{HunNac05}. Therefore, according to the bounded linear
transformation theorem~\cite[Theorem~5.19]{HunNac05},
$D^{\beta}|_{\partial\Omega}$ has a unique bounded linear
extension operator
$B^{\beta}$ on $\Hil^m(\Omega)$ such that
\[
B^{\beta}f=D^{\beta}|_{\partial\Omega}f=D^{\beta}f|_{\partial\Omega},
~f\in\Cont^{m}(\overline{\Omega})\quad{}
\textrm{and}\quad{}\norm{B^{\beta}f}_{\partial\Omega}\leq
C_{\beta}\norm{f}_{m,\Omega},~f\in\Hil^m(\Omega).
\]
We will call $B^{\beta}:\Hil^m(\Omega)\rightarrow\Leb_2(\partial\Omega)$ the \emph{boundary trace mapping of the $\beta^{\text{th}}$ derivative $D^{\beta}$}.
We use the convention for the notations $D^{\beta}|_{\partial\Omega}$ same as $B^{\beta}$ in this article.
%
%///////////////////////////////////////////////////////////////////////////////////////////////////////////////////////
\begin{remark}\label{r:Boundary}
The construction and definition of these boundary trace mappings are the same as in~\cite{AdaFou03,HunNac05}. In these references it is further shown that $D^{\beta}|_{\partial\Omega}$ is a surjective mapping from $\Hil^m(\Omega)$ onto $\Hil^{m-\abs{\beta}-1/2}(\partial\Omega)$ whenever $d\geq2$. However, we will not be concerned with the space $\Hil^{m-\abs{\beta}-1/2}(\partial\Omega)$ in this paper.
\end{remark}
%///////////////////////////////////////////////////////////////////////////////////////////////////////////////////////

When $d=1$ we also denote $\Cont(\partial\Omega):=\{f:\partial\Omega=\{a,b\}\rightarrow\RR\}$. So $\Cont(\partial\Omega)\subset\Leb_2(\partial\Omega)$
for every dimension $d\in\NN$ which implies that $b_{\beta}\circ D^{\beta}|_{\partial\Omega}f:=b_{\beta}(D^{\beta}|_{\partial\Omega}f)\in\Leb_2(\partial\Omega)$ when
$b_{\beta}\in\Cont(\partial\Omega)$ and $f\in\Hil^m(\Omega)$.
Furthermore $b_{\beta}\circ D^{\beta}|_{\partial\Omega}$ is continuous on $\Hil^m(\Omega)$.

%///////////////////////////////////////////////////////////////////////////////////////////////////////////////////////
\begin{definition}\label{d:bound}
A \emph{boundary operator} (with non-constant coefficients)
$B:\Hil^m(\Omega)\rightarrow\Leb_2(\partial\Omega)$ is well-defined
by
\[
B=\sum_{\abs{\beta}\leq
m-1}b_{\beta}\circ D^{\beta}|_{\partial\Omega},\quad{}\text{where
}b_{\beta}\in\Cont(\partial\Omega)\text{ and
}\beta\in\NN_0^d,~m\in\NN.
\]
The \emph{order} of $B$ is given by
\[
\Order(B):=\max\left\{\abs{\beta}:b_{\beta}\not\equiv0,~\abs{\beta}\leq
m-1,~\beta\in\NN_0^d\right\}.
\]
A \emph{vector boundary operator} $\vB=(B_1,\cdots,B_{\nb})^T$ is
formed using a finite number of boundary operators
$B_1,\cdots,B_{\nb}$ and its \emph{order} is
$\Order(\vB):=\max \{ \Order(B_1),\cdots,\Order(B_{\nb}) \}$.
\end{definition}
%///////////////////////////////////////////////////////////////////////////////////////////////////////////////////////

We can use the vector boundary operator $\vB=(B_1,\cdots,B_{\nb})^T$ of
order $m-1$ to define a \emph{$\vB$-semi-inner product} on
$\Hil^m(\Omega)$ via the form
\[
(f,g)_{\vB,\partial\Omega}=\sum_{j=1}^{\nb}(B_jf,B_jg)_{\partial\Omega},\quad{}
f,g\in\Hil^m(\Omega).
\]

Given a function $f\in\Hil^1(\Omega)$, it is well known that
$f\in\Hil^1_0(\Omega)$ if and only if $f$ vanishes on its boundary
trace. Therefore we need sufficiently many homogeneous boundary
conditions to determine whether a function $f\in\Hil^m(\Omega)$
belongs to $\Hil^m_0(\Omega)$.

%///////////////////////////////////////////////////////////////////////////////////////////////////////////////////////
\begin{definition}\label{d:B-semi-inner-product}
$\Bm$ is defined to be a collection of vector boundary operators $\vB=(B_1,\cdots,B_{\nb})^T$ of
order $m-1\in\NN_0$ which satisfy the requirement that for each fixed $f\in\Hil^m(\Omega)$
\[
\vB f=\v0\text{ if and only if
}D^{\beta}|_{\partial\Omega}f=0\text{ for each
}\abs{\beta}\leq m-1\text{ and }\beta\in\NN_0^d.
\]
\end{definition}
%///////////////////////////////////////////////////////////////////////////////////////////////////////////////////////
We illustrate Definition~\ref{d:B-semi-inner-product} with some examples for the set $\mathscr{B}^2_{\Omega}$ in the case $d=1$ with $\partial\Omega:=\{0,1\}$. Two possible members of $\mathscr{B}^2_{\Omega}$ are
\[
\vB_1=\begin{pmatrix}\frac{d}{dx}|_{\partial\Omega}\\I|_{\partial\Omega}\end{pmatrix}
\quad\text{or}\quad
\vB_2=\begin{pmatrix}\frac{d}{dx}|_{\partial\Omega}+I|_{\partial\Omega}\\
\frac{d}{dx}|_{\partial\Omega}-I|_{\partial\Omega}\end{pmatrix}.
\]
While these are both first-order vector boundary operators, their $\vB_1$ and $\vB_2$-semi-inner products defined in $\Hil^2(\Omega)$ are different.

Because of the trivial traces theorem~\cite[Theorem~5.37]{AdaFou03} we know that $f\in\Hil^m_0(\Omega)$ if and only if $D^{\beta}|_{\partial\Omega}f=0$
for each $\abs{\beta}\leq m-1$ and $\beta\in\NN_0^d$ whenever $f\in\Hil^m(\Omega)$. In analogy to this, we can verify the same trivial trace property for
the vector boundary operators $\vB\in\Bm$.
%///////////////////////////////////////////////////////////////////////////////////////////////////////////////////////
\begin{lemma}\label{l:B-semi-norm}
If $\vB\in\Bm$, then $f\in\Hil^m(\Omega)$ belongs to
$\Hil^m_0(\Omega)$ if and only if $\vB f=\v0$.
\end{lemma}
%///////////////////////////////////////////////////////////////////////////////////////////////////////////////////////

%---------------------------------------------------------------------------------------------------------------------
%---------------------------------------------------------------------------------------------------------------------

\subsection{Constructing Hilbert Spaces by Differential and Boundary Operators}\label{s:HpbAOmega}

Let $\Omega$ be a regular bounded open domain of $\Rd$. We want to
observe the relationship between our differential and boundary operators. Given a vector differential operator and a vector boundary operator, i.e.,
\[
\vP=(P_1,\cdots,P_{\np})^T\in\Pm,\quad{}\vB=(B_1,\cdots,B_{\nb})^T\in\Bm,\quad{}
m>d/2\text{ and }m\in\NN,
\]
the differential operator $L$ of order $\Order(L)=2m$ is
well-defined by
\[
L=\vP^{\ast T}\vP=\sum_{j=1}^{\np}P_j^{\ast}P_j.
\]
Next we can construct homogeneous differential equations with respect to $L$ and $\vB$ in the Sobolev space $\Hil^m(\Omega)$,
i.e.,
%/////////////////////////////////////////////////////////////////////////////////////////////////////////////////////
\begin{equation}\label{e:L-B}
\begin{cases}
Lf=0,&\text{in }\Omega,\\
\vB f=\v0,&\text{on }\partial\Omega.
\end{cases}
\end{equation}
%/////////////////////////////////////////////////////////////////////////////////////////////////////////////////////
Combining Equation~(\ref{e:L-B}) and the following Lemma~\ref{l:P-B-Eqn}, we will be able to
verify that the inner product spaces $\HpOmega$ and $\HbAOmega$ defined below are
well-defined (see Definitions~\ref{d:HpOmega} and~\ref{d:HbAOmega}).

%///////////////////////////////////////////////////////////////////////////////////////////////////////////////////////
\begin{lemma}\label{l:P-B-Eqn}
Equation~(\ref{e:L-B})
has the unique trivial solution $f\equiv0$ in $\Hil^m(\Omega)$.
\end{lemma}
%///////////////////////////////////////////////////////////////////////////////////////////////////////////////////////
\begin{proof}
It is obvious that $f\equiv0$ is a solution of Equation~(\ref{e:L-B}).
Suppose that $f\in\Hil^m(\Omega)$ is a solution of
Equation~(\ref{e:L-B}). Since $\vB\in\Bm$ and $\vB f=\v0$,
Lemma~\ref{l:B-semi-norm} tells us that $f\in\Hil^m_0(\Omega)$. Thus
there is a sequence $\{\gamma_k\}_{k=1}^{\infty}\subset\DOmega$ such
that $\norm{\gamma_k-f}_{m,\Omega}\rightarrow0$ when
$k\rightarrow\infty$. And then, using the two bilinear forms introduced earlier,
\[
\sum_{j=1}^{\np}(P_jf,P_jf)_{\Omega}
=\lim_{k\rightarrow\infty}\sum_{j=1}^{\np}(P_jf,P_j\gamma_k)_{\Omega}
=\lim_{k\rightarrow\infty}\sum_{j=1}^{\np}\langle P_j^{\ast}P_jf,\gamma_k
\rangle_{\Omega} =\lim_{k\rightarrow\infty}\langle Lf,\gamma_k
\rangle_{\Omega}=0.
\]
Since $\vP\in\Pm$, the generalized Sobolev inequality of Lemma~\ref{l:P-semi-norm} provides the estimate
\[
\norm{f}_{\Omega}^2\leq\norm{f}_{m,\Omega}^2\leq
C_P\abs{f}_{\vP,\Omega}^2=C_P\sum_{j=1}^{\np}\norm{P_jf}_{\Omega}^2=0,\quad{}C_P>0.
\]
This, however, implies that $f\equiv0$ is the unique solution of Equation~(\ref{e:L-B}).

\qed
\end{proof}
%/////////////////////////////////////////////////////////////////////////////////////////////////////////////////////

Note that in the above proof we employed both the integral and dual bilinear forms. Since we can only ensure that $P_j^{\ast}P_jf\in\dualDOmega$, this quantity needs to be handled with the dual bilinear form.
On the other hand, $P_jf\in\Leb_2(\Omega)$ implies that we can apply the integral bilinear form in this case. Using the notation introduced in~(\ref{e:DistNotation}),
we therefore obtain that $(P_jf,P_j\gamma_k)_{\Omega}=\langle P_jf,P_j\gamma_k\rangle_{\Omega}=\langle P_j^{\ast}P_jf,\gamma_k
\rangle_{\Omega}$ because $P_j\gamma_k\in\DOmega$.

%///////////////////////////////////////////////////////////////////////////////////////////////////////////////////////
\begin{definition}\label{d:HpOmega}
\[
\HpOmega:=\left\{f\in\Hil^m(\Omega):\vB f=\v0\right\},
\]
and it is equipped with the inner product
\[
(f,g)_{\HpOmega}:=(f,g)_{\vP,\Omega}=\sum_{j=1}^{\np}(P_jf,P_jg)_{\Omega},\quad{}
f,g\in\HpOmega.
\]
\end{definition}
%///////////////////////////////////////////////////////////////////////////////////////////////////////////////////////
We now show that the $\HpOmega$-inner product is well-defined. If
$f\in\HpOmega$ such that $\norm{f}_{\HpOmega}=0$, then $\vB f=\v0$ and
$\norm{P_jf}_{\Omega}=0$, $j=1,\cdots,\np$, which implies that
\[
\langle Lf,\gamma\rangle_{\Omega}= \sum_{j=1}^{\np}\langle
P^{\ast}_jP_jf,\gamma \rangle_{\Omega}
=\sum_{j=1}^{\np}(P_jf,P_j\gamma)_{\Omega}=
\sum_{j=1}^{\np}(0,P_j\gamma)_{\Omega}=0,
\quad{}\gamma\in\DOmega.
\]
Thus $f$ solves Equation~(\ref{e:L-B}) and then
Lemma~\ref{l:P-B-Eqn} states that $f=0$.

%///////////////////////////////////////////////////////////////////////////////////////////////////////////////////////
\begin{theorem}\label{t:Hp}
$\HpOmega$ and $\Hil^m_0(\Omega)$ are isomorphic, and therefore $\HpOmega$ is a separable Hilbert space.
\end{theorem}
%///////////////////////////////////////////////////////////////////////////////////////////////////////////////////////
\begin{proof}
Because of Lemma~\ref{l:B-semi-norm}, $\HpOmega=\Hil^m_0(\Omega)$.
The generalized Poincar\'e (Sobolev) inequality of Lemma~\ref{l:P-semi-norm} further shows that the $\HpOmega$-norm and
the $\Hil^m(\Omega)$-norm are equivalent on the space
$\Hil^m_0(\Omega)$.

\qed
\end{proof}
%/////////////////////////////////////////////////////////////////////////////////////////////////////////////////////

In Section~\ref{s:RKHS_Sob} we will establish relationships between $\HpOmega$ and Green kernels with homogeneous boundary conditions.
Furthermore, we will consider Green kernels with nonhomogeneous boundary conditions. To this end we need to define the inner product spaces $\HpbAOmega$ defined below.

%/////////////////////////////////////////////////////////////////////////////////////////////////////////////////////
\begin{definition}\label{d:HbAOmega}
Let the pair $\Aset:=\{\psi_k;a_k\}_{k=1}^{\na}$ for some $\na\in\NN_0$ where $\{a_k\}_{k=1}^{\na}\subset\RR^+$ and
$\{\psi_k\}_{k=1}^{\na}\subset\Null(L):=\{f\in\Hil^m(\Omega):Lf=0\}$ is an orthonormal subset with respect to the $\vB$-semi-inner product, i.e.,
$(\psi_k,\psi_l)_{\vB,\Omega}=\delta_{kl}$, a Kronecker delta function, $k,l=1,\cdots,\na$. Denote that
\[
\HbAOmega:=\Span\{\psi_1,\cdots,\psi_{\na}\}
\]
and it is equipped with the inner-product
\[
(f,g)_{\HbAOmega}:=\sum_{k=1}^{\na}\frac{\hat{f}_k\hat{g}_k}{a_k},\quad{}f,g\in\HbAOmega,
\]
where $\hat{f}_k$ and $\hat{g}_k$ are the Fourier coefficients of
$f$ and $g$ for the given orthonormal subset, i.e.,
\[
f=\sum_{k=1}^{\na}\hat{f}_k\psi_k,~g=\sum_{k=1}^{\na}\hat{g}_k\psi_k\quad{}
\text{and}\quad{}\{\hat{f}_k\}_{k=1}^{\na},\{\hat{g}_k\}_{k=1}^{\na}\subset\RR.
\]
In particular, if $\na=0$ or $\Aset:=\{0;0\}$ then $\HbAOmega:=\{0\}$ and $(0,0)_{\HbAOmega}:=0$.
\end{definition}
%/////////////////////////////////////////////////////////////////////////////////////////////////////////////////////
According to Lemma~\ref{l:P-B-Eqn}, the $\vB$-semi-inner product becomes an inner product on $\Null(L)$ which implies that the $\HbAOmega$-inner product is well-defined. It is obvious that $\HbAOmega$ is a separable Hilbert space which is embedded in the Sobolev space $\Hil^m(\Omega)$ because it is finite-dimensional.

We have now finally arrived at the definition we will use in our construction of reproducing kernel Hilbert spaces connected to Green kernels with nonhomogeneous boundary conditions.

%///////////////////////////////////////////////////////////////////////////////////////////////////////////////////////
\begin{definition}\label{d:HpbAOmega}
The direct sum space $\HpbAOmega$ is defined as
\[
\HpbAOmega:=\HpOmega\oplus\HbAOmega,
\]
and it is equipped with the inner product
\[
(f,g)_{\HpbAOmega}:=(f_P,g_P)_{\HpOmega}+(f_B,g_B)_{\HbAOmega},\quad{}
f,g\in\HpbAOmega,
\]
where $f_P,g_P\in\HpOmega$ and $f_B,g_B\in\HbAOmega$ are the unique
decompositions of $f,g$, i.e.,
\[
f=f_P+f_B,\quad{}g=g_P+g_B,\quad{}\text{where
}f_P,g_P\in\HpOmega\text{ and }f_B,g_B\in\HbAOmega.
\]
\end{definition}
%///////////////////////////////////////////////////////////////////////////////////////////////////////////////////////
The direct sum space $\HpbAOmega$ is well-defined because $\HpOmega\cap\Null(L)=\{0\}$.

%///////////////////////////////////////////////////////////////////////////////////////////////////////////////////////
\begin{theorem}\label{t:HpbA}
$\HpbAOmega$ is a separable Hilbert space and it is embedded in
$\Hil^m(\Omega)$. Moreover,
\[
(f,g)_{\HpbAOmega}=(f,g)_{\vP,\Omega}
+\sum_{k=1}^{\na}\frac{\hat{f}_k\hat{g}_k}{a_k}
-\sum_{k=1}^{\na}\sum_{l=1}^{\na}\hat{f}_k\hat{g}_l(\psi_k,\psi_l)_{\vP,\Omega},\quad{}f,g\in\HpbAOmega,
\]
where
\[
\hat{f}_k:=(f,\psi_k)_{\vB,\partial\Omega},\quad{}
\hat{g}_k:=(g,\psi_k)_{\vB,\partial\Omega},\quad{}k=1,\cdots,\na.
\]
In particular, if $\Aset=\left\{\psi_k;a_k\right\}_{k=1}^{\na}$
further satisfies $\{\psi_k\}_{k=1}^{\na}\subseteq\Null(\vP)$ then
\[
\norm{f}_{\HpbAOmega}^2=\abs{f}_{\vP,\Omega}^2+\sum_{k=1}^{\na}\frac{\abs{\hat{f}_k}^2}{a_k},\quad{}f\in\HpbAOmega.
\]
\end{theorem}
%///////////////////////////////////////////////////////////////////////////////////////////////////////////////////////
\begin{proof}
Since $\HpOmega$ and $\HbAOmega$ are separable Hilbert spaces which are embedded in $\Hil^m(\Omega)$, we can immediately
verify that $\HpbAOmega$ is a separable Hilbert space and that it is
embedded in $\Hil^m(\Omega)$.

Fix any $f=f_P+f_B\in\HpbAOmega$, where $f_P\in\HpOmega$ and
$f_B\in\HbAOmega$. We immediately have $\vB f_P=\v0$ and $Lf_B=0$. Since
$f_P\in\HpOmega\cong\Hil_0^m(\Omega)$, there is a sequence
$\{\gamma_k\}_{k=1}^{\infty}\subset\DOmega$ such that
$\norm{\gamma_k-f_P}_{m,\Omega}\rightarrow0$ when
$k\rightarrow\infty$. Thus we have
\[
\begin{split}
(f_B,f_P)_{\vP,\Omega}
&=\lim_{k\rightarrow\infty}\sum_{j=1}^{\np}(P_jf_B,P_j\gamma_k)_{\Omega}
=\lim_{k\rightarrow\infty}\sum_{j=1}^{\np}\langle
P_jf_B,P_j\gamma_k\rangle_{\Omega}\\
&=\lim_{k\rightarrow\infty}\sum_{j=1}^{\np}\langle
P_j^{\ast}P_jf_B,\gamma_k\rangle_{\Omega}=\lim_{k\rightarrow\infty}\langle
Lf_B,\gamma_k\rangle_{\Omega}=0.
\end{split}
\]
Because of $\vB f=\vB f_P+\vB f_B=\vB f_B$, we can compute the Fourier coefficients of $f$ as
$\hat{f}_k=(f,\psi_k)_{\vB,\partial\Omega}=(f_B,\psi_k)_{\vB,\partial\Omega}$
which implies that $f_B=\sum_{k=1}^{\na}\hat{f}_k\psi_k$ and
$\norm{f_B}_{\HbAOmega}^2=\sum_{k=1}^{\na}a_k^{-1}\abs{\hat{f}_k}^2$.
Since
\[
(f_B,f_B)_{\vP,\Omega}=\sum_{j=1}^{\np}(P_jf_B,P_jf_B)_{\Omega}
=\sum_{k=1}^{\na}\sum_{l=1}^{\na}\hat{f}_k\hat{f}_l\sum_{j=1}^{\np}(P_j\psi_k,P_j\psi_l)_{\Omega},
\]
we have
\[
(f,f)_{\vP,\Omega}=(f_P,f_P)_{\vP,\Omega}+2(f_P,f_B)_{\vP,\Omega}+(f_B,f_B)_{\vP,\Omega}=
(f_P,f_P)_{\vP,\Omega}+\sum_{k=1}^{\na}\sum_{l=1}^{\na}\hat{f}_k\hat{f}_l(\psi_k,\psi_l)_{\vP,\Omega}.
\]

Summarizing the above discussion, we obtain that
\[
\norm{f}_{\HpbAOmega}^2=\norm{f_P}_{\HpOmega}^2+\norm{f_B}_{\HbAOmega}^2
=\abs{f}_{\vP,\Omega}^2+\sum_{k=1}^{\na}\frac{\abs{\hat{f}_k}^2}{a_k}
-\sum_{k=1}^{\na}\sum_{l=1}^{\na}\hat{f}_k\hat{f}_l(\psi_k,\psi_l)_{\vP,\Omega}.
\]

$\qed$
\end{proof}
%///////////////////////////////////////////////////////////////////////////////////////////////////////////////////////

We can also check that $\HpbAOmega\cong\Hil^m_0(\Omega)\oplus\Span\{\psi_k\}_{k=1}^{\na}$, where the direct sum space is defined by the $\Hil^m(\Omega)$-norm.

%///////////////////////////////////////////////////////////////////////////////////////////////////////////////////////
\begin{corollary}\label{c:HpbAOmega-Null}
If $\Null(\vP)$ is finite-dimensional, then there is a
pair $\Aset$ as in Definition~\ref{d:HbAOmega} such that
$\HpbAOmega\cong\Hil^m_0(\Omega)\oplus\Null(\vP)$
with its inner product equal to
\[
(f,g)_{\HpbAOmega}=(f,g)_{\vP,\Omega}+(f,g)_{\vB,\partial\Omega},\quad{}f,g\in\HpbAOmega.
\]
(Here the direct sum space $\Hil^m_0(\Omega)\oplus\Null(\vP)$ is given the $\Hil^m(\Omega)$-norm.)
\end{corollary}
%///////////////////////////////////////////////////////////////////////////////////////////////////////////////////////

%///////////////////////////////////////////////////////////////////////////////////////////////////////////////////////
\begin{remark}\label{r:countableAset}
In~\cite{Ye12} the finite pair $\Aset=\{\psi_k;a_k\}_{k=1}^{\na}$ is generalized to a countable pair
$\Aset=\{\psi_k;a_k\}_{k=1}^{\infty}\subset\Null(L)\otimes\RR^+$
such that the $\HpbAOmega\cong\Hil^m(\Omega)$.
\end{remark}
%///////////////////////////////////////////////////////////////////////////////////////////////////////////////////////

%///////////////////////////////////////////////////////////////////////////////////////////////////////////////////////
\begin{corollary}\label{c:HpbOmega}
$\Hil^m_0(\Omega)\oplus\Null(L)=\Hil^m(\Omega)$.
\end{corollary}
%///////////////////////////////////////////////////////////////////////////////////////////////////////////////////////
To achieve the proof, we first show that $\Null(L)$ is complete with respect to the $\Hil^m(\Omega)$-norm.
For each $f\in\Hil^m(\Omega)$ we can find its orthogonal projection $f_P$ in $\Hil^m_0(\Omega)$ with respect to the $\vP$-semi-inner product.
Finally, we can check that $f_B:=f-f_P\in\Null(L)$. The complete proof is worked out in the thesis~\cite{Ye12}.

%---------------------------------------------------------------------------------------------------------------------
%/////////////////////////////////////////////////////////////////////////////////////////////////////////////////////
%---------------------------------------------------------------------------------------------------------------------

\section{Constructing Reproducing Kernels via Green
Kernels}\label{s:RKHS_Sob}

%---------------------------------------------------------------------------------------------------------------------
%---------------------------------------------------------------------------------------------------------------------

Let $\Omega$ be a regular bounded open domain of $\Rd$. Given a
vector differential operator $\vP=(P_1,\cdots,P_{\np})^T\in\Pm$ and a
vector boundary operator $\vB=(B_1,\cdots,B_{\nb})^T\in\Bm$, where $m>d/2$ and $m\in\NN$, we want to find a Green
kernel of the differential operator $L=\vP^{\ast
T}\vP=\sum_{j=1}^{\np}P_j^{\ast}P_j$ with either homogeneous or nonhomogeneous boundary conditions
given by $\vB$ so that it is also the reproducing kernel of a reproducing-kernel Hilbert space.
Furthermore, we assume that the pair
$\Aset:=\left\{\psi_k;a_k\right\}_{k=1}^{\na}\subset\Null(L)\otimes\RR^+$ satisfies the
conditions of Definition~\ref{d:HbAOmega} such that $\{\psi_k\}_{k=1}^{\na}$ is an orthonormal subset with respect to the $\vB$-semi-inner product.

In this section, we will show that the Green kernels with either homogeneous or nonhomogeneous boundary conditions
are reproducing kernels and that their reproducing-kernel Hilbert spaces can be represented by $\vP$, $\vB$ and $\Aset$.

%/////////////////////////////////////////////////////////////////////////////////////////////////////////////////////
\begin{definition}\label{d:Green-Nonhomo}
Suppose that the set
$\Rset:=\{\vGamma(\cdot,\vy):\vy\in\Omega\}\subseteq\otimes_{j=1}^{\nb}\Leb_2(\partial\Omega)$.
A kernel
$\Phi:\Omega\times\Omega\rightarrow\RR$ is called a \emph{Green kernel
of $L$ with boundary conditions given by $\vB$ and $\Rset$} if for each fixed
$\vy\in\Omega$, $\Phi(\cdot,\vy)\in\Hil^m(\Omega)$ is a solution of
\[
\begin{cases}
L\Phi(\cdot,\vy)=\delta_{\vy},&\text{in }\Omega,\\
\vB \Phi(\cdot,\vy)=\vGamma(\cdot,\vy),&\text{on }\partial\Omega.
\end{cases}
\]
If $\Rset\equiv\{0\}$, then the kernel $G:\Omega\times\Omega\rightarrow\RR$ is called a \emph{Green
kernel of $L$ with homogeneous boundary conditions given by $\vB$}, i.e.,
for each fixed
$\vy\in\Omega$, $G(\cdot,\vy)\in\Hil^m(\Omega)$ is a solution of
\[
\begin{cases}
LG(\cdot,\vy)=\delta_{\vy},&\text{in }\Omega,\\
\vB G(\cdot,\vy)=\v0,&\text{on }\partial\Omega.
\end{cases}
\]
(We can also use Lemma~\ref{l:P-B-Eqn} to show that the Green kernel is a unique solution.)
\end{definition}
%/////////////////////////////////////////////////////////////////////////////////////////////////////////////////////

Next we will view the relationship between the eigenvalues and eigenfunctions
of the Green kernels (reproducing kernels) and those of
the differential operators with either homogeneous or nonhomogeneous boundary conditions.

%/////////////////////////////////////////////////////////////////////////////////////////////////////////////////////
\begin{definition}\label{d:W-eig}
Let $\Phi\in\Leb_2(\Omega\times\Omega)$.
$\{\lambda_p\}_{p=1}^{\infty}\subset\RR$ and $\{e_p\}_{p=1}^{\infty}\subset\Leb_2(\Omega)\backslash\{0\}$
are called \emph{eigenvalues and eigenfunctions of $\Phi$} if for each fixed
$p\in\NN$,
\[
(\Integral_{\Phi,\Omega}e_p)(\vy)=(\Phi(\cdot,\vy),e_p)_{\Omega}=\lambda_pe_p(\vy),\quad{}\vy\in\Omega,
\]
where $\Integral_{\Phi,\Omega}$ is the integral operator defined in (\ref{e:IntOp}).
\end{definition}
%/////////////////////////////////////////////////////////////////////////////////////////////////////////////////////

%/////////////////////////////////////////////////////////////////////////////////////////////////////////////////////
\begin{definition}\label{d:L-eig-non}
Let the set
$\Eset:=\{\veta_p\}_{p=1}^{\infty}\subseteq\otimes_{j=1}^{\nb}\Leb_2(\partial\Omega)$.
$\{\mu_p\}_{p=1}^{\infty}\subset\RR$ and $\{e_p\}_{p=1}^{\infty}\subset\Hil^m(\Omega)\backslash\{0\}$
are called \emph{eigenvalues and eigenfunctions
of $L$ with boundary conditions given by $\vB$ and $\Eset$} if
for each fixed $p\in\NN$ we have
\[
\begin{cases}
Le_p=\mu_pe_p,&\text{in }\Omega,\\
\vB e_p=\veta_p,&\text{on }\partial\Omega.
\end{cases}
\]
If $\Eset\equiv\{0\}$, then $\{\mu_p\}_{p=1}^{\infty}\subset\RR$ and $\{e_p\}_{p=1}^{\infty}\subset\Hil^m(\Omega)\backslash\{0\}$
are called \emph{eigenvalues and eigenfunctions
of $L$ with homogeneous boundary conditions given by
$\vB$}, i.e., for each $p\in\NN$
\[
\begin{cases}
Le_p=\mu_pe_p,&\text{in }\Omega,\\
\vB e_p=\v0,&\text{on }\partial\Omega.
\end{cases}
\]
\end{definition}
%/////////////////////////////////////////////////////////////////////////////////////////////////////////////////////

The reader may be wondering about our use of different names for Green kernels. In the following we will use these different names to distinguish between a various types of Green kernels. The kernels $G$ and $K$ are defined in Theorems~\ref{t:RKHS-HpOmega} and~\ref{t:RKHS-HpbAOmega}, and they are Green kernels with homogeneous and nonhomogeneous boundary conditions respectively. Moreover, a kernel $R$ determined by the set $\Aset$ is introduced in Theorem~\ref{t:RKHS-HbAOmega}. We will verify below that $K$, $G$ and $R$ are reproducing kernels. Finally, we use the symbol $\Phi$ to denote the Green kernel corresponding to the general boundary conditions stated in Definition~\ref{d:Green-Nonhomo}. The Green kernel $\Phi$ may not be a reproducing kernel. An example of such a typical case is given in Remark~\ref{r:Green-non}.

%---------------------------------------------------------------------------------------------------------------------
%---------------------------------------------------------------------------------------------------------------------

\subsection{Green Kernels with Homogeneous Boundary
Conditions}\label{s:Green-Homo}

%/////////////////////////////////////////////////////////////////////////////////////////////////////////////////////
\begin{theorem}\label{t:RKHS-HpOmega}
Suppose that there is a Green kernel $G$ of $L$ with homogeneous boundary
conditions given by $\vB$ as in Definition~\ref{d:Green-Nonhomo}. Then $G$ is the reproducing kernel of the reproducing-kernel Hilbert space $\HpOmega$ (see Definition~\ref{d:HpOmega}) and $\HpOmega\cong\Hil^m_0(\Omega)$.
\end{theorem}
%/////////////////////////////////////////////////////////////////////////////////////////////////////////////////////
\begin{proof}
According to Theorem~\ref{t:Hp}, $\HpOmega\cong\Hil^m_0(\Omega)$.
Fix any $\vy\in\Omega$. Since $G(\cdot,\vy)\in\Hil^m(\Omega)$ and $\vB
G(\cdot,\vy)=\v0$, we have $G(\cdot,\vy)\in\HpOmega$ by Lemma~\ref{l:B-semi-norm}.

We now verify the reproducing property of $G$. According to the Sobolev
embedding theorem~\cite{AdaFou03}, $\Hil^m(\Omega)$ is embedded into
$\Cont(\overline{\Omega})$ when $m>d/2$, i.e., there is a positive
constant $C_m$ such that
\[
\norm{f}_{\Cont(\overline{\Omega})}:=\sup\left\{\abs{f(\vx)}:\vx\in\Omega\right\}
\leq C_m\norm{f}_{m,\Omega},\quad{}f\in\Hil^m(\Omega)\subseteq\Cont(\overline{\Omega}).
\]
For any fixed $f\in\HpOmega$ there is a sequence
$\left\{\gamma_k\right\}_{k=1}^{\infty}\subset\DOmega$ such that
\begin{equation}\label{e:RKHS-HpOmega-1}
\abs{f(\vy)-\gamma_k(\vy)}\leq\norm{f-\gamma_k}_{\Cont(\overline{\Omega})}
\leq C_m\norm{f-\gamma_k}_{m,\Omega}\rightarrow0,\quad{}\text{when
}k\rightarrow\infty.
\end{equation}
Since
\[
\begin{split}
&(G(\cdot,\vy),\gamma_k)_{\HpOmega}=\sum_{j=1}^{\np}(P_jG(\cdot,\vy),P_j\gamma_k)_{\Omega}
=\sum_{j=1}^{\np}\langle P_jG(\cdot,\vy),P_j\gamma_k\rangle_{\Omega}\\
=&\sum_{j=1}^{\np}\langle
P^{\ast}_jP_jG(\cdot,\vy),\gamma_k\rangle_{\Omega}= \langle
LG(\cdot,\vy),\gamma_k\rangle_{\Omega}=\langle\delta_{\vy},\gamma_k\rangle_{\Omega}=\gamma_k(\vy),\quad{}k\in\NN,
\end{split}
\]
we can determine that
\begin{equation}\label{e:RKHS-HpOmega-2}
\begin{split}
&\abs{(G(\cdot,\vy),f)_{\HpOmega}-\gamma_k(\vy)}=
\abs{(G(\cdot,\vy),f)_{\HpOmega}-(G(\cdot,\vy),\gamma_k)_{\HpOmega}}\\
\leq&\norm{f-\gamma_k}_{\HpOmega}\norm{G(\cdot,\vy)}_{\HpOmega} \leq
C_P\norm{f-\gamma_k}_{m,\Omega}\norm{G(\cdot,\vy)}_{m,\Omega}\rightarrow0,
\text{ when }k\rightarrow\infty,
\end{split}
\end{equation}
where the positive constant $C_P$ is independent of the function $f$. Here -- as before -- the two notations $(\cdot,\cdot)_{\Omega}$ and $\langle\cdot,\cdot\rangle_{\Omega}$ denote the integral bilinear form and the dual bilinear form, respectively (see Section~\ref{s:DiffAdjoint}).
Combining Equations~(\ref{e:RKHS-HpOmega-1})
and~(\ref{e:RKHS-HpOmega-2}), we will get
\[
(G(\cdot,\vy),f)_{\HpOmega}=f(\vy).
\]

$\qed$
\end{proof}
%/////////////////////////////////////////////////////////////////////////////////////////////////////////////////////

%/////////////////////////////////////////////////////////////////////////////////////////////////////////////////////
\begin{corollary}\label{c:PDK-HpOmega}
$G$ is a symmetric positive definite kernel on $\Omega$.
\end{corollary}
%/////////////////////////////////////////////////////////////////////////////////////////////////////////////////////
\begin{proof}
Fix any set of distinct points $X=\left\{\vx_1,\cdots,\vx_N\right\}\subset\Omega$ and coefficients
$\vc=\left(c_1,\cdots,c_N\right)^T\in\RR^N$, $N\in\NN$. Since $G$ is
the reproducing kernel of the reproducing kernel Hilbert space
$\HpOmega$, $G$ is symmetric and positive semi-definite, i.e.,
\[
\sum_{j=1}^{N}\sum_{k=1}^{N}c_jc_kG(\vx_j,\vx_k)=
(\sum_{j=1}^Nc_jG(\cdot,\vx_j),\sum_{k=1}^Nc_kG(\cdot,\vx_k))_{\HpOmega}
=\norm{\sum_{j=1}^Nc_jG(\cdot,\vx_j)}_{\HpOmega}^2\geq0.
\]
To get strict positive definiteness we assume $\sum_{j=1}^Nc_jG(\cdot,\vx_j)=0$. For any $\gamma\in\DOmega$,
\[
\sum_{j=1}^Nc_j\gamma(\vx_j)=
\sum_{j=1}^Nc_j\langle\delta_{\vx_j},\gamma\rangle_{\Omega}
=\sum_{j=1}^Nc_j\langle LG(\cdot,\vx_j),\gamma\rangle_{\Omega}
=(\sum_{j=1}^Nc_jG(\cdot,\vx_j),\gamma)_{\vP,\Omega}=0.
\]
To show that $c_j=0$, $j=1,\cdots,N$, we pick an arbitrary $\vx_j \in X$ and construct $\gamma_j \in \DOmega$ such that
$\gamma_j$ vanishes on $X\backslash\{\vx_j\}$, but
$\gamma_j(\vx_j)\neq0$. Therefore
\[
\sum_{j=1}^N\sum_{k=1}^{N}c_jc_kG(\vx_j,\vx_k)>0,\quad{}\text{when
}\vc\neq0.
\]

$\qed$
\end{proof}
%/////////////////////////////////////////////////////////////////////////////////////////////////////////////////////

Since
$G(\cdot,\vy)\in\Cont(\overline{\Omega})$ for each $\vy\in\Omega$, $G$ is
uniformly continuous on $\Omega$ which implies that $G\in\Leb_2(\Omega\times\Omega)$. According to Mercer's
theorem~\cite[Theorem~13.5]{Fas07}, there is an orthonormal basis
$\{e_p\}_{p=1}^{\infty}$ of $\Leb_2(\Omega)$ and a positive sequence
$\{\lambda_p\}_{p=1}^{\infty}$ such that
$G(\vx,\vy)=\sum_{p=1}^{\infty}\lambda_pe_p(\vx)e_p(\vy)$ and
$(G(\cdot,\vy),e_p)_{\Omega}=\lambda_pe_p(\vy)$, $\vx,\vy\in\Omega$,
$p\in\NN$. 
According to Proposition~\ref{p:RKHS-Ltwo}, we can use the technology of the proof of \cite[Proposition~10.29]{Wen05} 
to verify $\{\sqrt{\lambda_{p}}e_p\}_{p=1}^{\infty}$ is an orthonormal basis
of $\HpOmega$. (We firstly show that $\{\sqrt{\lambda_{p}}e_p\}_{p=1}^{\infty}$ is an orthonormal subset
of $\HpOmega$. Next we can verify that it is complete.)

%/////////////////////////////////////////////////////////////////////////////////////////////////////////////////////
\begin{Proposition}\label{p:W-eig-L-B-eig}
If $\{\lambda_p\}_{p=1}^{\infty}\subset\RR^{+}$ and $\{e_p\}_{p=1}^{\infty}$ are
the eigenvalues and eigenfunctions of $G$, then
$\{\lambda_p^{-1}\}_{p=1}^{\infty}$ and $\{e_p\}_{p=1}^{\infty}$ are
the eigenvalues and eigenfunctions of $L$ with homogeneous boundary
conditions given by $\vB$. Moreover,
$\{\sqrt{\lambda_{p}}e_p\}_{p=1}^{\infty}$ is an orthonormal basis
of $\HpOmega$ whenever $\{e_p\}_{p=1}^{\infty}$
is an orthonormal basis of $\Leb_2(\Omega)$.
\end{Proposition}
%/////////////////////////////////////////////////////////////////////////////////////////////////////////////////////
%
\begin{proof}
According to Fubini's theorem~\cite[Theorem~12.41]{HunNac05}, for each fixed
$p\in\NN$ and any $\gamma\in\DOmega$,
\[
\begin{split}
&\langle Le_p,\gamma \rangle_{\Omega}=(e_p,L^{\ast}\gamma)_{\Omega}
=\int_{\Omega}e_p(\vy)(L^{\ast}\gamma)(\vy)\ud\vy\\
=&\int_{\Omega}\lambda_p^{-1}(G(\cdot,\vy),e_p)_{\Omega}(L^{\ast}\gamma)(\vy)\ud\vy
=\int_{\Omega}\int_{\Omega}\lambda_p^{-1}G(\vx,\vy)e_p(\vx)(L^{\ast}\gamma)(\vy)\ud\vx\ud\vy\\
=&\int_{\Omega}\lambda_p^{-1}e_p(\vx)\left(G(\vx,\cdot),L^{\ast}\gamma\right)_{\Omega}\ud\vx
=\int_{\Omega}\lambda_p^{-1}e_p(\vx)\langle G(\cdot,\vx),L^{\ast}\gamma\rangle_{\Omega}\ud\vx\\
=&\int_{\Omega}\lambda_p^{-1}e_p(\vx)\langle
LG(\cdot,\vx),\gamma\rangle_{\Omega}\ud\vx
=\int_{\Omega}\lambda_p^{-1}e_p(\vx)\langle\delta_{\vx},\gamma\rangle_{\Omega}\ud\vx\\
=&\int_{\Omega}\lambda_p^{-1}e_p(\vx)\gamma(\vx)\ud\vx
=\langle\lambda_p^{-1}e_p,\gamma\rangle_{\Omega}.
\end{split}
\]
This shows that $Le_p=\lambda_{p}^{-1}e_p$.

According to Proposition~\ref{p:RKHS-Ltwo}, the integral operator $\Integral_{G,\Omega}$
is a continuous map from $\Leb_2(\Omega)$ to $\HpOmega$. Since
$\lambda_pe_p(\vy)=(G(\cdot,\vy),e_p)_{\Omega}=(\Integral_{G,\Omega}e_p)(\vy)$,
$\vy\in\Omega$, we can conclude that $e_p\in\HpOmega$. This implies that $\vB
e_p=\v0$, $p\in\NN$. Therefore $\{\lambda_p^{-1}\}_{p=1}^{\infty}$ and $\{e_p\}_{p=1}^{\infty}$
are the eigenvalues and eigenfunctions
of $L$ with homogeneous boundary conditions given by $\vB$.

\qed
\end{proof}
%/////////////////////////////////////////////////////////////////////////////////////////////////////////////////////

%/////////////////////////////////////////////////////////////////////////////////////////////////////////////////////
\begin{Proposition}\label{p:L-B-eig-W-eig}
If $\{\mu_p\}_{p=1}^{\infty}\subset\RR^{+}$ and $\{e_p\}_{p=1}^{\infty}$ are the eigenvalues and eigenfunctions
of $L$ with homogeneous boundary conditions given by
$\vB$, then $\{\mu_p^{-1}\}_{p=1}^{\infty}$ and $\{e_p\}_{p=1}^{\infty}$
are the eigenvalues and eigenfunctions
of $G$. Moreover, if $\{e_p\}_{p=1}^{\infty}$ is
an orthonormal basis of $\Leb_2(\Omega)$, then
\[
G(\vx,\vy)=\sum_{p=1}^{\infty}\mu_p^{-1}e_p(\vx)e_p(\vy),\quad{}
\vx,\vy\in\Omega.
\]
\end{Proposition}
%/////////////////////////////////////////////////////////////////////////////////////////////////////////////////////
\begin{proof}
According to Theorem~\ref{t:RKHS-HpOmega} $G$ is a reproducing kernel, i.e., we have
\[
(G(\cdot,\vy),e_p)_{\HpOmega}=e_p(\vy),\quad{}\vy\in\Omega,\quad{}p\in\NN.
\]
Applying the same method as in Equation~(\ref{e:W-eig-L-B-eig}), we obtain
\[
(G(\cdot,\vy),e_p)_{\HpOmega}=\sum_{j=1}^{\np}(P_jG(\cdot,\vy),P_je_p)_{\Omega}
=(G(\cdot,\vy),\mu_pe_p)_{\Omega}.
\]
Combining the above equations, we can easily verify that
$(G(\cdot,\vy),e_p)_{\Omega}=\mu_p^{-1}e_p(\vy)$. The second claim follows immediately.

\qed
\end{proof}
%/////////////////////////////////////////////////////////////////////////////////////////////////////////////////////

%----------------------------------------------------------------------------------------------------------
%----------------------------------------------------------------------------------------------------------

\subsection{Green Kernels with Nonhomogeneous Boundary Conditions}\label{s:Green-Nonhomo}

%/////////////////////////////////////////////////////////////////////////////////////////////////////////////////////
\begin{theorem}\label{t:RKHS-HbAOmega}
The space $\HbAOmega$ of Definition~\ref{d:HbAOmega} is a reproducing-kernel Hilbert space with
reproducing kernel
\[
R(\vx,\vy):=\sum_{k=1}^{\na}a_k\psi_k(\vx)\psi_k(\vy),
\quad{}\vx,\vy\in\Omega.
\]
In particular, when $\na=0$ or $\Aset=\{0;0\}$ then $R:=0$.
\end{theorem}
%/////////////////////////////////////////////////////////////////////////////////////////////////////////////////////
\begin{proof}
We fix any $\vy\in\Omega$. It is obvious that $R(\cdot,\vy)=\sum_{k=1}^{\na}(a_k\psi_k(\vy))\psi_k\in\HbAOmega$.

We now turn to the reproducing property. Let any $f=\sum_{k=1}^{\na}\hat{f}_k\psi_k\in\HbAOmega$. Then
\[
(R(\cdot,\vy),f)_{\HbAOmega}=\sum_{k=1}^{\na}\frac{a_k\psi_k(\vy)\hat{f}_k}{a_k}
=\sum_{k=1}^{\na}\hat{f}_k\psi_k(\vy)=f(\vy),\quad{}\vy\in\Omega.
\]

$\qed$
\end{proof}
%/////////////////////////////////////////////////////////////////////////////////////////////////////////////////////

Our main theorem now follows directly from Theorems~\ref{t:HpbA},~\ref{t:RKHS-HpOmega}
and~\ref{t:RKHS-HbAOmega}.
%/////////////////////////////////////////////////////////////////////////////////////////////////////////////////////
\begin{theorem}\label{t:RKHS-HpbAOmega}
Suppose that there is a Green kernel $G$ of $L$ with homogeneous boundary
conditions given by $\vB$. Then the direct sum space $\HpbAOmega$ (see Definition~\ref{d:HpbAOmega}) is a reproducing-kernel Hilbert space with
reproducing kernel
\[
K(\vx,\vy):=G(\vx,\vy)+R(\vx,\vy),\quad{}\vx,\vy\in\Omega.
\]
Moreover, $\HpbAOmega$ can be embedded into $\Hil^m(\Omega)$.
\end{theorem}
%/////////////////////////////////////////////////////////////////////////////////////////////////////////////////////

By Corollary~\ref{c:PDK-HpOmega} we know that $G$ is a symmetric positive definite kernel, and using similar arguments we can check that $R$ is symmetric positive semi-definite. Together, this allows us to formulate the following corollary.
%/////////////////////////////////////////////////////////////////////////////////////////////////////////////////////
\begin{corollary}\label{c:PDK-HpbAOmega}
$K$ is a symmetric positive definite kernel on $\Omega$.
\end{corollary}
%/////////////////////////////////////////////////////////////////////////////////////////////////////////////////////
On the other hand, $K$ may not be positive definite on $\partial\Omega$ (see the min kernel in Example~\ref{exa_Min}).
According to Definition~\ref{d:Green-Nonhomo} we also have
%/////////////////////////////////////////////////////////////////////////////////////////////////////////////////////
\begin{corollary}\label{c:Green-HpbAOmega}
Let $\Rset:=\left\{\vB R(\cdot,\vy):\vy\in\Omega\right\}$.
Then $K$ is a Green kernel of $L$ with boundary conditions
given by $\vB$ and $\Rset$.
\end{corollary}
%/////////////////////////////////////////////////////////////////////////////////////////////////////////////////////
%
%/////////////////////////////////////////////////////////////////////////////////////////////////////////////////////
\begin{remark}\label{r:Green-non}
To see that not every Green kernel is a reproducing kernel, assume that $\Phi$ is a Green kernel of the differential operator $L$. Then, according to Corollary~\ref{c:HpbOmega},
$\Phi$ can be uniquely written in the form
\[
\Phi(\vx,\vy)=\Phi_P(\vx,\vy)+\Phi_B(\vx,\vy),\quad{}
\Phi_P(\cdot,\vy)\in\Hil_0^m(\Omega),~\Phi_B(\cdot,\vy)\in\Null(L),\quad{}\vx,\vy\in\Omega.
\]
Therefore we have
\[
\begin{cases}
L\Phi_P(\cdot,\vy)=\delta_{\vy},&\text{in }\Omega,\\
\vB \Phi_P(\cdot,\vy)=\v0,&\text{on }\partial\Omega,
\end{cases}
\quad{}\text{and}\quad{}
\begin{cases}
L\Phi_B(\cdot,\vy)=0,&\text{in }\Omega,\\
\vB \Phi_B(\cdot,\vy)=\vB\Phi(\cdot,\vy),&\text{on }\partial\Omega.
\end{cases}
\]
This means that $\Phi_P$ is a Green kernel of $L$ with homogeneous
boundary conditions given by $\vB$. However, there may be no pair $\Aset$
such that $R=\Phi_B$ even though $\Aset$ is extended to a countable pair set. This shows that $\Phi$ may not be a reproducing kernel of a
reproducing-kernel Hilbert space. For example,
$\Phi(x,y):=-\frac{1}{2}\abs{x-y}$ is the Green kernel of
$L:=-\frac{d^2}{dx^2}$. However, $\phi(x):=\Phi(x,0)$ is only a conditionally
positive definite function of order one and therefore cannot be a reproducing kernel.
\end{remark}
%/////////////////////////////////////////////////////////////////////////////////////////////////////////////////////

We are now ready to address nonhomogeneous boundary conditions.
Consider a kernel
$\Gamma\in\Leb_2(\partial\Omega\times\Omega)$. Then we can define an
\emph{integral operator}
$\Integral_{\Gamma,\Omega}:\Leb_2(\Omega)\rightarrow\Leb_2(\partial\Omega)$
via the form
\[
(\Integral_{\Gamma,\Omega}f)(\vx):=(\Gamma(\vx,\cdot),f)_{\Omega},\quad{}f\in\Leb_2(\Omega)\text{
and }\vx\in\partial\Omega.
\]

Let $\vGamma$ denote the vector function
$\vGamma(\cdot,\vy)=(\Gamma_1(\cdot,\vy),\cdots,\Gamma_{\nb}(\cdot,\vy))^T:=\vB
K(\cdot,\vy)$ for any $\vy\in\Omega$, i.e., $\Gamma_j(\cdot,\vy)=B_jK(\cdot,\vy)$, $j=1,\cdots,\nb$.
Since $B_jG(\cdot,\vy)=0$, $\vy\in\Omega$, we have
\[
\Gamma_j(\cdot,\vy)=B_jK(\cdot,\vy)=B_jG(\cdot,\vy)+B_jR(\cdot,\vy)
=B_jR(\cdot,\vy)=\sum_{k=1}^{\na}a_k(B_j\psi_k)\psi_k(\vy).
\]
As a consequence we have $\Gamma_j\in\Leb_2(\partial\Omega\times\Omega)$.
%/////////////////////////////////////////////////////////////////////////////////////////////////////////////////////
\begin{Proposition}\label{p:W-eig-L-B-eig-non}
%If a sequence $\{\lambda_p\}_{p=1}^{\infty}\subset\RR^+$ and an orthonormal basis $\{e_p\}_{p=1}^{\infty}$ of $\Leb_2(\Omega)$ are the eigenvalues and eigenfunctions of $K$, then
%$\{\lambda_p^{-1}\}_{p=1}^{\infty}$ and $\{e_p\}_{p=1}^{\infty}$ are
%the eigenvalues and eigenfunctions of $L$ with boundary conditions given by $\vB$ and
%\[
%\Eset:=\{\veta_p:=(\lambda_p^{-1}\Integral_{\Gamma_1,\Omega}e_p,\cdots,
%\lambda_p^{-1}\Integral_{\Gamma_{\nb},\Omega}e_p)^T\}_{p=1}^{\infty},
%\]
%i.e., $\eta_{p,j}(\vx)=\lambda_p^{-1}(\Gamma_j(\vx,\cdot),e_p)_{\Omega}$, $\vx\in\partial\Omega$.
%Moreover,
%$\{\sqrt{\lambda_p}e_p\}_{p=1}^{\infty}$ is an orthonormal basis of
%$\HpbAOmega$.
If $\{\lambda_p\}_{p=1}^{\infty}\subset\RR^+$ and $\{e_p\}_{p=1}^{\infty}$ are the eigenvalues and eigenfunctions of $K$, then
$\{\lambda_p^{-1}\}_{p=1}^{\infty}$ and $\{e_p\}_{p=1}^{\infty}$ are
the eigenvalues and eigenfunctions of $L$ with boundary conditions given by $\vB$ and
\[
\Eset:=\{\veta_p:=(\lambda_p^{-1}\Integral_{\Gamma_1,\Omega}e_p,\cdots,
\lambda_p^{-1}\Integral_{\Gamma_{\nb},\Omega}e_p)^T\}_{p=1}^{\infty},
\]
i.e., $\eta_{p,j}(\vx)=\lambda_p^{-1}(\Gamma_j(\vx,\cdot),e_p)_{\Omega}$, $\vx\in\partial\Omega$. 
Moreover,
$\{\sqrt{\lambda_{p}}e_p\}_{p=1}^{\infty}$ is an orthonormal basis
of $\HpbAOmega$ whenever $\{e_p\}_{p=1}^{\infty}$
is an orthonormal basis of $\Leb_2(\Omega)$.
\end{Proposition}
%/////////////////////////////////////////////////////////////////////////////////////////////////////////////////////
\begin{proof}
Using the same method as in the proof of
Proposition~\ref{p:W-eig-L-B-eig}, we can verify that $\langle
Le_p,\gamma \rangle_{\Omega}=\langle \lambda_p^{-1}e_p,\gamma
\rangle_{\Omega}$ for each $\gamma\in\DOmega$. This implies that
$Le_p=\lambda_p^{-1}e_p$, $p\in\NN$.

Next we compute their boundary conditions. Fix any boundary operator $B_j$, $j=1,\cdots,\nb$ and any eigenfunction $e_p$ and eigenvalue $\lambda_p$ of $K$, $p\in\NN$.
Because $K\in\Cont(\overline{\Omega}\times\overline{\Omega})$ is positive definite. According to Mercer's Theorem, there exist an orthonormal basis $\{\varphi_k\}_{k=1}^{\infty}$ of $\Leb_2(\Omega)$ and
a positive sequence $\{\nu_k\}_{k=1}^{\infty}$ such that $K(\vx,\vy)=\sum_{k=1}^{\infty}\nu_k\varphi_k(\vx)\varphi_k(\vy)$, $\vx,\vy\in\Omega$.
We can also check that
$\{\sqrt{\nu_k}\varphi_k\}_{k=1}^{\infty}$ is an orthonormal basis of $\HpbAOmega$. 
Let $K_n(\vx,\vy):=\sum_{k=1}^{n}\nu_k\varphi_k(\vx)\varphi_k(\vy)$, $n\in\NN$. Thus
$\norm{K(\cdot,\vy)-K_n(\cdot,\vy)}_{\HpbAOmega}^2=\sum_{k=n+1}^{\infty}\nu_k\abs{\varphi_k(\vy)}^2\rightarrow0$ when $n\rightarrow\infty$.
According to Theorem~\ref{t:HpbA}, $\HpbAOmega$ is embedded into $\Hil^m(\Omega)$, which implies that $\norm{K(\cdot,\vy)-K_n(\cdot,\vy)}_{m,\Omega}\rightarrow0$ when $n\rightarrow\infty$.
So $B_jK(\cdot,\vy)=\sum_{k=1}^{\infty}\nu_k(B_j\varphi_k)\varphi_k(\vy)$ and 
$(B_{j,\vx}K(\vx,\cdot),e_p)_{\Omega}=\sum_{k=1}^{\infty}\nu_k(B_j\varphi_k)(\vx)(\varphi_k,e_p)_{\Omega}$. It implies that
\[
\lambda_p(B_{j}e_p)(\vx)
=B_{j,\vx}(K(\vx,\cdot),e_p)_{\Omega}
=(B_{j,\vx}K(\vx,\cdot),e_p)_{\Omega}=(\Gamma_j(\vx,\cdot),e_p)_{\Omega},
\quad \vx\in\partial\Omega.
\]

It follows that the boundary conditions have the form $\vB e_p=\veta_p$ for all $p\in\NN$.

\qed
\end{proof}
%/////////////////////////////////////////////////////////////////////////////////////////////////////////////////////

%/////////////////////////////////////////////////////////////////////////////////////////////////////////////////////
\begin{Proposition}\label{p:L-B-eig-W-eig-non}
%If a sequence $\{\mu_p\}_{p=1}^{\infty}\subset\RR^+$ and an orthonormal basis $\{e_p\}_{p=1}^{\infty}$ of $\Leb_2(\Omega)$ are the eigenvalues and eigenfunctions of
%$L$ with boundary conditions given by $\vB$ and
%\[
%\Eset:=\{\veta_p:=(\mu_p\Integral_{\Gamma_1,\Omega}e_p,\cdots,\mu_p\Integral_{\Gamma_{\nb},\Omega}e_p)^T\}_{p=1}^{\infty},
%\]
%i.e., $\eta_{pj}(\vx)=\mu_p(\Gamma_j(\vx,\cdot),e_p)_{\Omega}$, $\vx\in\partial\Omega$, then $\{\mu_p^{-1}\}_{p=1}^{\infty}$ and $\{e_p\}_{p=1}^{\infty}$
%are the eigenvalues and eigenfunctions
%of $K$. Moreover,
%\[
%K(\vx,\vy)=\sum_{p=1}^{\infty}\mu_p^{-1}e_p(\vx)e_p(\vy),
%\quad{}\vx,\vy\in\Omega.
%\]
If $\{\mu_p\}_{p=1}^{\infty}\subset\RR^+$ and $\{e_p\}_{p=1}^{\infty}$ of $\Leb_2(\Omega)$ are the eigenvalues and eigenfunctions of
$L$ with boundary conditions given by $\vB$ and
\[
\Eset:=\{\veta_p:=(\mu_p\Integral_{\Gamma_1,\Omega}e_p,\cdots,\mu_p\Integral_{\Gamma_{\nb},\Omega}e_p)^T\}_{p=1}^{\infty},
\]
i.e., $\eta_{pj}(\vx)=\mu_p(\Gamma_j(\vx,\cdot),e_p)_{\Omega}$, $\vx\in\partial\Omega$, then $\{\mu_p^{-1}\}_{p=1}^{\infty}$ and $\{e_p\}_{p=1}^{\infty}$
are the eigenvalues and eigenfunctions
of $K$. Moreover, if $\{e_p\}_{p=1}^{\infty}$ is
an orthonormal basis of $\Leb_2(\Omega)$, then
\[
K(\vx,\vy)=\sum_{p=1}^{\infty}\mu_p^{-1}e_p(\vx)e_p(\vy),
\quad{}\vx,\vy\in\Omega.
\]
\end{Proposition}
%/////////////////////////////////////////////////////////////////////////////////////////////////////////////////////
\begin{proof}
We fix any $p\in\NN$. Let $v_p(\vy):=\mu_p(R(\cdot,\vy),e_p)_{\Omega}=\mu_p\sum_{k=1}^{\na}a_k(\psi_k,e_p)_{\Omega}\psi_k(\vy)$,
$\vy\in\Omega$. Then $Lv_p=0$ and $\vB v_p=\veta_p$ because $\vB K(\cdot,\vy)=\vB R(\cdot,\vy)$ for each $\vy\in\Omega$.

Define $u_p:=e_p-v_p$, so that $Lu_p=Le_p=\mu_pe_p$ and $\vB u_p=\vB e_p-\vB u_p=0$ which implies that $u_p\in\HpOmega$.
As in Proposition~\ref{p:L-B-eig-W-eig}, we can obtain that
\[
(G(\cdot,\vy),\mu_p
e_{p})_{\Omega}=(G(\cdot,\vy),Lu_p)_{\Omega}=(G(\cdot,\vy),u_p)_{\HpOmega}=
u_p(\vy),\quad{}\vy\in\Omega.
\]
It follows from the above discussion that
\[
(K(\cdot,\vy),e_p)_{\Omega}=(G(\cdot,\vy),e_p)_{\Omega}+(R(\cdot,\vy),e_p)_{\Omega}
=\mu_p^{-1}u_p(\vy)+\mu_p^{-1}v_p(\vy)=\mu_p^{-1}e_p(\vy),\quad{}\vy\in\Omega.
\]

$\qed$
\end{proof}
%/////////////////////////////////////////////////////////////////////////////////////////////////////////////////////

Given a function $f\in\Hil^m(\Omega)$, we also want to know whether
$f$ belongs to the reproducing kernel Hilbert space $\HpbAOmega$
as used in Theorem~\ref{t:RKHS-HpbAOmega}. According to
Corollary~\ref{c:HpbOmega}, $f$ can be uniquely decomposed into
$f=f_P+f_B$, where $f_P\in\HpOmega$ and $f_B\in\Null(L)$.
Theorem~\ref{t:HpbA} shows that $f\in\HpbAOmega$ if and only if
$f_B\in\HbAOmega$. Moreover, $f_B\in\HbAOmega$ if and only if
$\sum_{k=1}^{\na}a_k^{-1}\abs{\hat{f}_k}^2<\infty$, where
$\hat{f}_k:=(f,\psi_k)_{\vB,\partial\Omega}$ for each $k\in\NN$.

Because $\sum_{k=1}^{\na}a_k\norm{\psi_k}_{m,\Omega}^2<\infty$. We can set
$\Psi_j(\vx,\vy):=B_{j,\vx}B_{j,\vy}R(\vx,\vy)$, $\vx,\vy\in\partial\Omega$ and $j=1,\cdots,\nb$. Then
$\Psi_j(\vx,\vy)=\sum_{k=1}^{\na}a_k(B_j\psi_k)(\vx)(B_j\psi_k)(\vy)$
which implies that each $\Psi_j$ is symmetric
positive semi-definite on $\partial\Omega$. So $\Psi_j$ is the reproducing
kernel of a reproducing-kernel Hilbert space
$\Hilbert_j(\partial\Omega)$ by \cite[Theorem~1.3.3]{BerThAg04}.
According to \cite[Theorem~10.29]{Wen05}, we have
$\sum_{k=1}^{\na}a_k^{-1}\abs{\hat{f}_k}^2<\infty$ if and only if
$B_jf\in\Hilbert_j(\partial\Omega)$, $j=1,\cdots,\nb$.

%/////////////////////////////////////////////////////////////////////////////////////////////////////////////////////
%\begin{corollary}\label{c:f-in-HpbAOmega}
\begin{theorem}\label{t:f-in-HpbAOmega}
Let $\Psi_j(\vx,\vy):=B_{j,\vx}B_{j,\vy}R(\vx,\vy)$, $\vx,\vy\in\partial\Omega$ and $j=1,\cdots,\nb$.
Use
$\Hilbert_j(\partial\Omega)$ to denote the reproducing-kernel Hilbert space
whose reproducing kernel is $\Psi_j$.
Then a function $f\in\Hil^m(\Omega)$ belongs to $\HpbAOmega$ if and
only if $B_jf\in\Hilbert_j(\partial\Omega)$ for each
$j=1,\cdots,\nb$.
\end{theorem}
%\end{corollary}
%/////////////////////////////////////////////////////////////////////////////////////////////////////////////////////

%/////////////////////////////////////////////////////////////////////////////////////////////////////////////////////
\begin{remark}
In Remark~\ref{r:countableAset} we mentioned that the nonhomogeneous boundary conditions discussed in the present paper can be generalized to such that are generated by a countable set $\Aset$. One will also want to know which Green kernels associated with such nonhomogeneous boundary conditions are reproducing kernels. In the thesis~\cite{Ye12} it is shown that,
e.g., a Green kernel $\Phi\in\Hil^{m,m}(\Omega\times\Omega)$ is a reproducing kernel if and only if $B_{j,\vx}B_{j,\vy}\Phi$ is positive semi-definite on $\partial\Omega$ for each $j=1,\cdots,\nb$.
This Green kernel can then be expanded as the sum of eigenvalues and eigenfunctions analogous to Propositions~\ref{p:W-eig-L-B-eig-non} and~\ref{p:L-B-eig-W-eig-non}.
This allows us to approximate the interpolant $s_{f,X}$ by a truncated expansion of the
Green kernel.
\end{remark}
%/////////////////////////////////////////////////////////////////////////////////////////////////////////////////////

%---------------------------------------------------------------------------------------------------------------------
%/////////////////////////////////////////////////////////////////////////////////////////////////////////////////////
%---------------------------------------------------------------------------------------------------------------------

\section{Examples}\label{s:exa}

%---------------------------------------------------------------------------------------------------------------------
%---------------------------------------------------------------------------------------------------------------------

%/////////////////////////////////////////////////////////////////////////////////////////////////////////////////////
\begin{example}[Modifications of the Min Kernel]\label{exa_Min}
Let
\[
\Omega:=(0,1),\quad{}\vP:=\frac{d}{dx},\quad{}L:=P^{\ast}_1P_1=-\frac{d^2}{dx^2},\quad{}\vB:=I|_{\partial\Omega}=I|_{\{0,1\}}.
\]
It is easy to check that $\vP\in\mathscr{P}_{\Omega}^{1}$ and $\vB\in\mathscr{B}_{\Omega}^{1}$, where $\Order(\vP)=\Order(\vB)+1=1>1/2$.
We can calculate the Green kernel $G$ of $L$ with homogeneous boundary
conditions given by $\vB$, i.e.,
\[
G(x,y):= \min\{x,y\}-xy,\quad{}x,y\in\Omega.
\]
This Green kernel $G$ is also known to be the covariance kernel of the Brownian bridge.
According to Theorem~\ref{t:RKHS-HpOmega}, $G$ is the reproducing
kernel of the reproducing-kernel Hilbert space
\[
\HpOmega=\left\{f\in\Hil^1(\Omega):f(0)=f(1)=0\right\}\cong\Hil_0^1(\Omega),
\]
with the inner product
\[
(f,g)_{\HpOmega}=(f,g)_{\vP,\Omega}=(f',g')_{\Omega}=\int_0^1f'(x)g'(x)\ud
x,\quad{}f,g\in\HpOmega.
\]

In order to obtain a second, related, kernel we consider the same differential operator
with a different set of \emph{nonhomogeneous} boundary conditions.
One of the obvious orthonormal subsets of $\Null(L)=\Span\{\psi_1,\psi_2\}$ with respect to the $\vB$-semi-inner product is given by
\[
\psi_1(x):=x,\quad{}\psi_2(x):=1-x,\quad{}x\in\Omega,
\]
and we can further obtain that
\[
\hat{f}_1:=(f,\psi_1)_{\vB,\partial\Omega}=f(1),\quad{}
\hat{f}_2:=(f,\psi_2)_{\vB,\partial\Omega}=f(0),\quad{}f\in\Hil^1(\Omega).
\]
We will choose the nonnegative coefficients
\[
a_1:=1,\quad{}a_2:=0,
\]
to set up the pair $\Aset:=\{\psi_k;a_k\}_{k=1}^2$. According to
Theorems~\ref{t:RKHS-HbAOmega} and~\ref{t:RKHS-HpbAOmega}, the
covariance kernel of the standard Brownian motion
\[
K(x,y)=G(x,y)+R(x,y)=G(x,y)+a_1\psi_1(x)\psi_1(y)
=\min\{x,y\},\quad{}x,y\in\Omega,
\]
is the reproducing kernel of the reproducing-kernel Hilbert space
\[
\HpbAOmega=\HpOmega\oplus\HbAOmega=\HpOmega\oplus
\Span\{\psi_1\}=\{f\in\Hil^1(\Omega):f(0)=0\},
\]
with the inner product
\[
(f,g)_{\HpbAOmega}=(f,g)_{\vP,\Omega}
+\frac{\hat{f}_1\hat{g}_1}{a_1}
-\hat{f}_1\hat{g}_1(\psi_1,\psi_1)_{\vP,\Omega}=\int_0^1f'(x)g'(x)\ud
x,\quad{}f,g\in\HpbAOmega.
\]

If we select another pair $\Aset$, i.e.,
\[
\psi_1(x):=\frac{\sqrt{2}}{2},\quad{}\psi_2(x):=\sqrt{2}x-\frac{\sqrt{2}}{2},\quad{}
a_1:=1,\quad{}a_2:=0,
\]
then we can deal with \emph{periodic} boundary conditions.
Thus we obtain the reproducing-kernel Hilbert space
\[
\HpbAOmega=\HpOmega\oplus\Null(\vP)=\HpOmega\oplus\Span\{\psi_1\}=\{f\in\Hil^1(\Omega):f(0)=f(1)\}
\]
equipped with the inner product
\[
(f,g)_{\HpbAOmega}=(f,g)_{\vP,\Omega}+(f,g)_{\vB,\partial\Omega}=\int_0^1f'(x)g'(x)\ud
x+f(0)g(0)+f(1)g(1),
\]
whose reproducing kernel has the form
\[
K(x,y):=G(x,y)+a_1\psi_1(x)\psi_1(y)=
\min\{x,y\}-xy+\frac{1}{2},\quad{}x,y\in\Omega.
\]

\end{example}
%/////////////////////////////////////////////////////////////////////////////////////////////////////////////////////

%/////////////////////////////////////////////////////////////////////////////////////////////////////////////////////
\begin{example}[Univariate Sobolev Splines]\label{exa_Sobolev}
Let $\sigma$ be a positive scaling parameter and
\[
\Omega:=(0,1),\quad{}\vP:=(\frac{d}{dx},\sigma I)^T,
\quad{}L_{\sigma}:=\sum_{j=1}^2P^{\ast}_jP_j=-\frac{d^2}{dx^2}+\sigma^2
I,\quad{}\vB:=I|_{\partial\Omega}.
\]
Then $\vP\in\mathscr{P}_{\Omega}^{1}$ and $\vB\in\mathscr{B}_{\Omega}^{1}$. So the Green kernel $G_{\sigma}$ of $L_\sigma$ with
homogeneous boundary conditions given by $\vB$ has the form
\[
G_{\sigma}(x,y):=
\begin{cases}
\frac{1}{\sigma\sinh(\sigma)}\sinh(\sigma x)\sinh(\sigma-\sigma y),&0<x\leq y<1,\\
\frac{1}{\sigma\sinh(\sigma)}\sinh(\sigma-\sigma x)\sinh(\sigma
y),&0<y\leq x<1.
\end{cases}
\]

Using the same approach as in Example~\ref{exa_Min} we can pick an orthonormal bases of $\Null(L)$ with respect to the $\vB$-semi-inner product as
\[
\begin{split}
&\psi_1(x):=\frac{\exp(\sigma-\sigma
x)}{\sqrt{2}\left(\exp(\sigma)-1\right)}-\frac{\exp(\sigma
x)}{\sqrt{2}\left(\exp(\sigma)-1\right)},\\
&\psi_2(x):=\frac{\exp(\sigma-\sigma
x)}{\sqrt{2}\left(\exp(\sigma)+1\right)}+\frac{\exp(\sigma
x)}{\sqrt{2}\left(\exp(\sigma)+1\right)},
\end{split}
\]
and then compute
\[
\hat{f}_1:=\left(f,\psi_1\right)_{\vB,\partial\Omega}=\frac{1}{\sqrt{2}}\left(f(0)-f(1)\right),\quad
\hat{f}_2:=\left(f,\psi_2\right)_{\vB,\partial\Omega}=\frac{1}{\sqrt{2}}\left(f(0)+f(1)\right).
\]
We further choose the positive sequence
\[
a_1:=\frac{\exp(\sigma)-1}{2\sigma\exp(\sigma)},\quad{}a_2:=\frac{\exp(\sigma)+1}{2\sigma\exp(\sigma)}.
\]
According to Theorem~\ref{t:RKHS-HpbAOmega},
\[
K(x,y)=G_{\sigma}(x,y)+R(x,y)=G_{\sigma}(x,y)+\sum_{k=1}^2a_k\psi_k(x)\psi_k(y)
=\frac{1}{2\sigma}\exp\left(-\sigma\abs{x-y}\right)
\]
is the reproducing kernel of the reproducing-kernel Hilbert space
$\HpbAOmega\cong\Hil^1(\Omega)$ with the inner-product
\[
\left(f,g\right)_{\HpbAOmega}=\int_0^1f'(x)g'(x)\ud
x+\sigma^2\int_0^1f(x)g(x)\ud x+2\sigma f(0)g(0)+2\sigma f(1)g(1).
\]

\end{example}
%/////////////////////////////////////////////////////////////////////////////////////////////////////////////////////
%
%/////////////////////////////////////////////////////////////////////////////////////////////////////////////////////
\begin{remark}
Roughly speaking, the differential operator $L_{\sigma}=-\frac{d^2}{dx^2}+\sigma^2
I$ converges to the operator $L=-\frac{d^2}{dx^2}$ from Example~\ref{exa_Min} when $\sigma\rightarrow0$. We also
observe that the homogeneous Green kernel $G_{\sigma}$ of
$L_{\sigma}$ converges uniformly to the homogeneous Green kernel $G$
of $L$ when $\sigma\rightarrow0$. This matter is discussed in detail for radial kernels of even smoothness orders in the paper~\cite{FasHicRidSong10}. One might hope to exploit this limiting behavior to stabilize the positive definite interpolation matrix corresponding to $G_\sigma$ when $\sigma$ is small by augmenting the matrix with polynomial blocks that correspond to the better-conditioned limiting kernel $G$.
\end{remark}
%/////////////////////////////////////////////////////////////////////////////////////////////////////////////////////

%---------------------------------------------------------------------------------------------------------------------
%---------------------------------------------------------------------------------------------------------------------

%/////////////////////////////////////////////////////////////////////////////////////////////////////////////////////
\begin{example}[Modifications of Thin Plate Splines]\label{exa:ThinPlate-2D}
Let $\Omega:=(0,1)^2\subset\RR^2$ and
\[
\vP:=(\frac{\partial^2}{\partial
x_1^2},~\sqrt{2}\frac{\partial^2}{\partial x_1\partial
x_2},~\frac{\partial^2}{\partial x_2^2})^T,\quad
\vB:=(\frac{\partial}{\partial
x_1}|_{\partial\Omega},~\frac{\partial}{\partial
x_2}|_{\partial\Omega},
~I|_{\partial\Omega})^T.
\]
which shows that $\vP\in\mathscr{P}_{\Omega}^2$ and
$\vB\in\mathscr{B}_{\Omega}^2$. Thus we can compute that
\[
L:=\sum_{j=1}^3P_j^{\ast}P_j=\Delta^2.
\]

We know that the fundamental solution of $L$ is given by
\[
\phi(\vx):=\frac{1}{8\pi}\norm{\vx}_2^2\log\norm{\vx}_2,\quad{}\vx\in\RR^2,
\]
i.e., $L\phi=\delta_0$ in $\RR^2$. Applying Green's formulas, we can
find a corrector function $\phi^{\vy}\in\Hil^2(\Omega)$ for each
fixed $\vy\in\Omega$ by solving
\[
\begin{cases}
L\phi^{\vy}=\Delta^2\phi^{\vy}=0,&\text{in }\Omega,\\
\vB\phi^{\vy}=\vGamma(\cdot,\vy),&\text{on }\partial\Omega,
\end{cases}
\]
where $\Gamma_1(\vx,\vy):=\frac{1}{8\pi}(2\log\norm{\vx-\vy}_2+1)(x_1-y_1)$, $\Gamma_2(\vx,\vy):=\frac{1}{8\pi}(2\log\norm{\vx-\vy}_2+1)(x_2-y_2)$
and $\Gamma_3(\vx,\vy):=\frac{1}{8\pi}\norm{\vx-\vy}_2^2\log\norm{\vx-\vy}_2$.
Since $\vGamma(\vx,\vy)=\vB_{\vx}\phi(\vx-\vy)$ for each $\vx\in\partial\Omega$ and $\vy\in\Omega$, the kernel
$G(\vx,\vy):=\phi(\vx-\vy)-\phi^{\vy}(\vx)$ defined in $\Omega\times\Omega$ is a
Green kernel of $L$ with homogeneous boundary conditions given by $\vB$.

Since $\Null(\vP)=\pi_1(\Omega)$, the space of linear polynomials on $\Omega$, we can obtain an
orthonormal basis of $\pi_1(\Omega)$ with respect to the $\vB$-semi-inner product as
\[
\psi_1(\vx):=\frac{1}{2},\quad{}\psi_2(\vx):=\sqrt{\frac{3}{29}}(x_1-2),
\quad{}\psi_3(\vx):=\sqrt{\frac{3}{29}}(x_2-2),\quad{}\vx:=(x_1,x_2)\in\Omega.
\]
We choose positive coefficients $\{a_k\}_{k=1}^{3}$ as $a_1=a_2=a_3:=1$.
Thus $R(\vx,\vy):=\sum_{k=1}^3a_k\psi_k(\vx)\psi_k(\vy)$.
According to Theorems~\ref{t:HpbA} and \ref{t:RKHS-HpbAOmega}, the
Green kernel
\[
K(\vx,\vy):=G(\vx,\vy)+R(\vx,\vy),\quad{}\vx,\vy\in\Omega,
\]
is the reproducing kernel of the reproducing-kernel Hilbert space
$\HpbAOmega=\Hil_0^m(\Omega)\oplus\pi_1(\Omega)$ and its inner-product has the form
\[
(f,g)_{\HpbAOmega}:=(f,g)_{\vP,\Omega}+(f,g)_{\vB,\partial\Omega},\quad{}
f,g\in\HpbAOmega.
\]

\cite[Chapters~10 and 11]{Wen05} state that the native space $\mathcal{N}_{\phi}(\Omega)$ of the thin plate spline $\phi$ covers the Sobolev space $\Hil^2(\Omega)$. Therefore $\HpbAOmega\varsubsetneqq\Hil^2(\Omega)\subseteq\mathcal{N}_{\phi}(\Omega)$.

\end{example}
%/////////////////////////////////////////////////////////////////////////////////////////////////////////////////////
%
%/////////////////////////////////////////////////////////////////////////////////////////////////////////////////////
\begin{remark}
We can also introduce other $d$-dimensional examples that connect Green kernels with, e.g., pdLg splines~\cite{FigChen90} or Sobolev splines~\cite{FasYe10}. A pdLg spline is given by a linear combination of the homogeneous Green kernel centered at the data sites from $X$. Thus it provides the $\vP$-semi-norm-optimal solution of the scattered data interpolation problem.
According to Example~5.7 of \cite{FasYe10}, the Mat\'ern function (or Sobolev spline) $\phi_{m,\sigma}$ of order $m>d/2$ with shape parameter $\sigma>0$ can be identified with the
kernel $\Phi_{m,\sigma}(\vx,\vy)=\phi_{m,\sigma}(\vx-\vy)$ which is a \emph{full-space} Green kernel of the differential operator $L:=(\Delta-\sigma I)^m$. If we add nonhomogeneous boundary conditions to $L$ then the finite set $\Aset$ used in the present paper does not allow us to discuss the resulting Green kernel $\Phi_{m,\sigma}$ and to check whether it is a reproducing kernel in a regular bounded open domain $\Omega$. This is done in the thesis~\cite{Ye12} where it is shown that for each $\sigma$ the reproducing-kernel Hilbert space associated with $\Phi_{m,\sigma}$ is equivalent to the Sobolev space $\Hil^m(\Omega)$. However, different shape parameters $\sigma$ allow us to choose a specific norm for $\Hil^m(\Omega)$ that reflects the relative influence of various derivatives in the data.
\end{remark}
%/////////////////////////////////////////////////////////////////////////////////////////////////////////////////////

%---------------------------------------------------------------------------------------------------------------------
%/////////////////////////////////////////////////////////////////////////////////////////////////////////////////////
%---------------------------------------------------------------------------------------------------------------------

\section{Acknowledgements}

The second author would like to express his gratitude to Dr.~K.~E.~Atkinson, who hosted Q. Ye at the University of
Iowa and provided valuable suggestions that allowed us to make significant improvements to this paper.

%---------------------------------------------------------------------------------------------------------------------
%/////////////////////////////////////////////////////////////////////////////////////////////////////////////////////
%---------------------------------------------------------------------------------------------------------------------

%% References

%%%%%%

\end{document}